\pdfoutput=1
\RequirePackage{ifpdf}
\ifpdf 
\documentclass[pdftex]{sigma}
\else
\documentclass{sigma}
\fi

\numberwithin{equation}{section}

\usepackage{mathrsfs,xy,xypic}

\newtheorem{thm}{Theorem}[section]
\newtheorem{lem}[thm]{Lemma}

\newtheorem{pro}[thm]{Proposition}
{\theoremstyle{definition}
\newtheorem{ex}[thm]{Example}
\newtheorem{rmk}[thm]{Remark}
\newtheorem{defi}[thm]{Definition}}

\newcommand{\lon }{\rightarrow}
\newcommand{\Hexp}{{\mathsf{Hexp}}}
\newcommand{\Diff}{{\mathsf{Diff}}}
\newcommand{\eb}{\mathfrak e_\beta}
\newcommand{\g}{\frkg}
\newcommand{\h}{\frkh}
\newcommand{\huaH}{\mathcal{H}}
\newcommand{\frke}{\mathfrak e}
\newcommand{\frkg}{\mathfrak g}
\newcommand{\frkh}{\mathfrak h}
\newcommand{\Id}{{\rm Id}}
\newcommand{\Hom}{\mathsf{Hom}}
\newcommand{\Der}{\mathsf{Der}}
\newcommand{\Ad}{\mathsf{Ad}}
\newcommand{\Aut}{\mathsf{Aut}}
\newcommand{\gl}{\mathfrak{gl}}
\newcommand{\sln}{\mathfrak {sl}}
\newcommand{\cen}{\mathsf{Cen}}
\newcommand{\ad}{\mathsf{ad}}

\begin{document}

\allowdisplaybreaks

\newcommand{\arXivNumber}{1904.06515}

\renewcommand{\PaperNumber}{137}

\FirstPageHeading

\ShortArticleName{Hom-Lie Algebras and Hom-Lie Groups, Integration and Differentiation}

\ArticleName{Hom-Lie Algebras and Hom-Lie Groups,\\ Integration and Differentiation}

\Author{Jun JIANG~$^\dag$, Satyendra Kumar MISHRA~$^\ddag$ and Yunhe SHENG~$^\dag$}

\AuthorNameForHeading{J.~Jiang, S.K.~Mishra and Y.~Sheng}

\Address{$^\dag$~Department of Mathematics, Jilin University, Changchun, Jilin Province, 130012, China}
\EmailD{\href{mailto:jiangjun20@mails.jlu.edu.cn}{jiangjun20@mails.jlu.edu.cn}, \href{mailto:shengyh@jlu.edu.cn}{shengyh@jlu.edu.cn}}

\Address{$^\ddag$~Statistics and Mathematics Unit, Indian Statistical Institute Bangalore, India}
\EmailD{\href{mailto:satyamsr10@gmail.com}{satyamsr10@gmail.com}}

\ArticleDates{Received June 01, 2020, in final form December 10, 2020; Published online December 17, 2020}

\Abstract{In this paper, we introduce the notion of a (regular) Hom-Lie group. We associate a Hom-Lie algebra to a Hom-Lie group and show that every regular Hom-Lie algebra is integrable. Then, we define a Hom-exponential ($\mathsf{Hexp}$) map from the Hom-Lie algebra of a Hom-Lie group to the Hom-Lie group and discuss the universality of this $\mathsf{Hexp}$ map. We also describe a Hom-Lie group action on a smooth manifold. Subsequently, we give the notion of an adjoint representation of a Hom-Lie group on its Hom-Lie algebra. At last, we integrate the Hom-Lie algebra $(\mathfrak{gl}(V),[\cdot,\cdot],\mathsf{Ad})$, and the derivation Hom-Lie algebra of a Hom-Lie algebra.}


\Classification{17B40; 17B61; 22E60; 58A32}

\section{Introduction}

The notion of a Hom-Lie algebra first appeared in the study of quantum deformations of Witt and Virasoro algebras in~\cite{HLS}. Hom-Lie algebras are generalizations of Lie algebras, where the Jacobi identity is twisted by a linear map, called the Hom-Jacobi identity. It is known that $q$-deformations of the Witt and the Virasoro algebras have the structure of a Hom-Lie algebra~\cite{HLS,hu}. There is a growing interest in Hom-algebraic structures because of their close relationship with the discrete and deformed vector fields, and differential calculus \cite{HLS,LD1,LD2}. In particular, representations and deformations of Hom-Lie algebras were studied in \cite{AEM,MS1,sheng}; a categorical interpretation of Hom-Lie algebras was described in \cite{CatDes}; the categorification of Hom-Lie algebras was given in~\cite{shengchen}; geometric and algebraic generalizations of Hom-Lie algebras were given in~\cite{CLS,LGT,HLR}; quantization of Hom-Lie algebras was studied in~\cite{Yao2}; and the universal enveloping algebra of Hom-Lie algebras was studied in~\cite{LGMT,Yau1}.

The notion of Hom-groups was initially introduced by Caenepeel and Goyvaerts in \cite{CatDes}. In \cite{LGMT}, Laurent-Gengoux, Makhlouf and Teles first gave a new construction of the universal enveloping algebra that is different from the one in \cite{Yau1}. This new construction leads to a Hom-Hopf algebra structure on the universal enveloping algebra of a Hom-Lie algebra. Moreover, one can associate a Hom-group to any Hom-Lie algebra by considering group-like elements in its universal enveloping algebra. Recently, M.~Hassanzadeh developed representations and a (co)homology theory for Hom-groups in \cite{HG1}. He also proved Lagrange's theorem for finite Hom-groups in \cite{HG2}. The recent developments on Hom-groups (see \cite{HG1,HG2,LGMT}) make it natural to study Hom-Lie groups and to explore the relationship between Hom-Lie groups and Hom-Lie algebras.

In this paper, we introduce a (real) Hom-Lie group as a Hom-group $(G,\diamond,e_{\Phi},\Phi)$, where the underlying set $G$ is a (real) smooth manifold, the Hom-group operations (such as the product and the inverse) are smooth maps, and the underlying structure map $\Phi\colon G\rightarrow G$ is a diffeomorphism. We associate a Hom-Lie algebra to a Hom-Lie group by considering the notion of left-invariant sections of the pullback bundle $\Phi^{!}TG$. We define one-parameter Hom-Lie subgroups of a Hom-Lie group and discuss a Hom-analogue of the exponential map from the Hom-Lie algebra of a Hom-Lie group to the Hom-Lie group. Later on, we consider Hom-Lie group actions on a~manifold $M$ with respect to a map $\iota\in \Diff(M)$ and define an adjoint representation of a Hom-Lie group on its Hom-Lie algebra. Finally, we discuss the integration of the Hom-Lie algebra $(\gl(V),[\cdot,\cdot]_{\beta},\Ad_{\beta})$ and the derivation Hom-Lie algebra $\Der(\g)$ of a Hom-Lie algebra $(\g,[\cdot,\cdot]_{\g},\phi_{\g})$.

All the results in the paper are under the regularity hypothesis. As appeared in the literature, e.g., \cite{LS}, Hom-Lie algebras are not necessarily regular. We will explore this more general case in the future.

The paper is organized as follows.
In Section \ref{sec:p}, we recall some basic definitions and results concerning Hom-Lie algebras and Hom-groups.
In Section \ref{Homliealgebra}, we define the notion of a Hom-Lie group with some useful examples. If $(G,\diamond,e_{\Phi},\Phi)$ is a Hom-Lie group, then we show that the space of left-invariant sections of the pullback bundle $\Phi^{!}TG$ has a Hom-Lie algebra structure. Consequently, we deduce a Hom-Lie algebra structure on the fibre of the pullback bundle $\Phi^{!}TG$ at~$e_{\Phi}$. In this way, we associate a Hom-Lie algebra $\big(\g^{!},[\cdot,\cdot]_{\g^{!}},\phi_{\g^{!}}\big)$ to the Hom-Lie group $(G,\diamond,e_{\Phi},\Phi)$, where $\g^{!}:=\Phi^{!}T_{e_{\Phi}}G$. We also show that every regular Hom-Lie algebra is integrable. Next, we define one-parameter Hom-Lie subgroups of a Hom-Lie group $(G,\diamond,e_{\Phi},\Phi)$ in terms of a weak homomorphism of Hom-Lie groups. We establish a one-to-one correspondence with $\g^{!}$. In the end, we define Hom-exponential map $\Hexp\colon \g^{!}\rightarrow G$ and discuss its properties.
In Section~\ref{actionhom}, we study Hom-Lie group actions on a smooth manifold $M$ with respect to a diffeomorphism $\iota\in \Diff(M)$. We define representations of a Hom-Lie group on a vector space~$V$ with respect to a map $\beta\in {\rm GL}(V)$, which leads to the notion of an adjoint representation of a Hom-Lie group on the associated Hom-Lie algebra.
In Section~\ref{Integrationgl}, we consider the Hom-Lie group $({\rm GL}(V),\diamond, \beta, \Ad_{\beta})$, where $V$ is a vector space and $\beta\in {\rm GL}(V)$. We show that the triple $(\gl(V),[\cdot,\cdot]_{\beta},\Ad_{\beta})$ is the Hom-Lie algebra of the Hom-Lie group $({\rm GL}(V),\diamond, \beta, \Ad_{\beta})$.
In the last section, we show that the derivation Hom-Lie algebra $\Der(\g)$ is the Hom-Lie algebra of the Hom-Lie group of automorphisms of $(\g,[\cdot,\cdot]_{\g},\phi_{\g})$.

\section{Preliminaries}\label{sec:p}

In this section, we first recall definitions of Hom-Lie algebras and Hom-groups.

\subsection{Hom-Lie algebras}
\begin{defi}
 A (multiplicative) {\it Hom-Lie algebra} is a triple $(\mathfrak{g},[\cdot,\cdot]_{\mathfrak{g}},\phi_{\mathfrak{g}})$ consisting of a~vector space $\mathfrak{g}$, a skew-symmetric bilinear map (bracket) $[\cdot,\cdot]_{\mathfrak{g}}\colon \wedge^2\mathfrak{g}\longrightarrow
 \mathfrak{g}$, and a linear map $\phi_{\mathfrak{g}}\colon \mathfrak{g}\lon \mathfrak{g}$ preserving the bracket, such that the following Hom-Jacobi
 identity with respect to $\phi_{\mathfrak{g}}$ is satisfied:
 \begin{equation*}
 [\phi_{\mathfrak{g}}(x),[y,z]_\g]_\g+[\phi_{\mathfrak{g}}(y),[z,x]_\g]_\g+[\phi_{\mathfrak{g}}(z),[x,y]_\g]_\g=0, \qquad\forall\, x,y,z\in \g.
 \end{equation*}
A Hom-Lie algebra $(\mathfrak{g},[\cdot,\cdot]_{\mathfrak{g}},\phi_{\mathfrak{g}})$ is called a {\it regular Hom-Lie algebra} if $\phi_{\mathfrak{g}}$ is
an invertible map.
\end{defi}

\begin{lem}\label{RemarkS6}
Let $(\g, [\cdot, \cdot]_{\g}, \phi_{\g})$ be a regular Hom-Lie algebra. Then $(\g, [\cdot, \cdot]_{\rm Lie})$ is a Lie algebra, where the Lie bracket $[\cdot, \cdot]_{\rm Lie}$ is given by $[x, y]_{\rm Lie}=\big[\phi_{\g}^{-1}(x), \phi_{\g}^{-1}(y)\big]_{\g}$ for all $x, y \in \g$.
\end{lem}

In the sequel, we always assume that $\phi_{\mathfrak{g}}$ is an invertible map. That is, in this paper, all the Hom-Lie algebras are assumed to be regular Hom-Lie algebras.

\begin{rmk}
In \cite{CatDes}, the authors used a monoidal categorical
approach to give an intrinsic study of regular Hom-type algebraic structures. In particular, a regular Hom-Lie algebra is called a~monoidal Hom-Lie algebra there. See~\cite{BoVer} for the categorical framework study of the BiHom-type structures.
\end{rmk}

\begin{defi} Let $(\mathfrak{g},[\cdot,\cdot]_{\mathfrak{g}},\phi_{\mathfrak{g}})$ and $ (\mathfrak{h},[\cdot,\cdot]_{\mathfrak{h}},\phi_{\mathfrak{h}})$ be Hom-Lie algebras.
\begin{itemize}\itemsep=0pt
 \item[\rm{(i)}] A linear map $f\colon \mathfrak{g}\lon \mathfrak{h}$ is called a {\it weak homomorphism} of Hom-Lie algebras if
\begin{equation*}
\phi_{\h}f[x,y]_{\g}=[f(\phi_{\g}(x)),f(\phi_{\g}(y))]_{\h},\qquad\forall\, x,y\in \mathfrak{g}.
\end{equation*}

 \item[\rm{(ii)}] A weak homomorphism $f\colon\mathfrak{g}\lon \mathfrak{h}$ is called a {\it homomorphism} if $f$ also satisfies
 \begin{gather*}
 f\circ \phi_{\mathfrak{g}} = \phi_{\mathfrak{h}}\circ f.
 \end{gather*}
\end{itemize}
 \end{defi}

\begin{defi}A representation of a Hom-Lie algebra $(\mathfrak{g},[\cdot,\cdot]_{\frkg},\phi_{\mathfrak{g}})$ on a vector space $V$ with respect to $\beta \in\mathfrak{gl}(V)$ is a linear map $\rho\colon \mathfrak{g}\lon \mathfrak{gl}(V)$ such that for all $x,y\in\mathfrak{g}$, the following equations are satisfied
\begin{gather*}
\rho(\phi_{\mathfrak{g}}(x))\circ \beta = \beta\circ\rho(x),\\
\rho([x,y]_{\mathfrak{g}})\circ\beta = \rho(\phi_{\mathfrak{g}}(x))\circ\rho(y)-\rho(\phi_{\mathfrak{g}}(y))\circ\rho(x).
\end{gather*}
\end{defi}

For all $x\in\mathfrak{g}$, let us define a map $\ad_{x}\colon \mathfrak{g}\lon \mathfrak{g}$ by
\begin{gather*}
\ad_{x}(y)=[x,y]_{\mathfrak{g}},\qquad\forall\, y \in \mathfrak{g}.
\end{gather*}
Then $\ad\colon \g\longrightarrow\gl(\mathfrak{g})$ is a representation of the Hom-Lie algebra $(\mathfrak{g},[\cdot,\cdot]_{\mathfrak{g}},\phi_{\mathfrak{g}})$ on $\g$ with respect to $\phi_\g$, which is called the {\it adjoint representation}.

Let $(\rho,V,\beta)$ be a representation of a Hom-Lie algebra $(\mathfrak{g},[\cdot,\cdot]_{\frkg},\phi_{\mathfrak{g}})$ on a vector space $V$ with respect to the map $\beta \in\mathfrak{gl}(V)$. Then let us recall from \cite{CaiSheng} that the cohomology of the Hom-Lie algebra $(\mathfrak{g},[\cdot,\cdot]_{\mathfrak{g}},\phi_\g)$ with coefficients in $(\rho,V,\beta)$ is the cohomology of the cochain complex $C^{k}(\mathfrak{g},V)=\Hom(\wedge^{k}\mathfrak{g},V)$ with the coboundary operator $d\colon C^{k}(\mathfrak{g},V)\lon C^{k+1}(\mathfrak{g},V)$ defined by
\begin{gather*}
 (df)(x_{1},\dots,x_{k+1})=\sum_{i=1}^{k+1}(-1)^{i+1}\rho(x_{i})\big(f\big(\phi_{\mathfrak{g}}^{-1}x_{1},\dots,
\widehat{\phi_{\mathfrak{g}}^{-1}x_{i}},\dots,\phi_{\mathfrak{g}}^{-1}x_{k+1}\big)\big)\\
\qquad{}+\sum_{i<j}(-1)^{i+j}\beta
f\big(\big[\phi_{\mathfrak{g}}^{-2}x_{i},\phi_{\mathfrak{g}}^{-2}x_{j}\big]_{\mathfrak{g}},\phi_{\mathfrak{g}}^{-1}x_{1},\dots,
\widehat{\phi_{\mathfrak{g}}^{-1}x_{i}},\dots,\widehat{\phi_{\mathfrak{g}}^{-1}x_{j}},\dots,\phi_{\mathfrak{g}}^{-1}x_{k+1}\big).
\end{gather*}
 The fact that $d^2=0$ is proved in \cite{CaiSheng}. Denote by $\mathcal{Z}^k(\g;\rho)$ and $\mathcal{B}^k(\g;\rho)$ the sets of $k$-cocycles and $k$-coboundaries respectively. We define the $k$-th cohomology group
$\mathcal{H}^k(\g;\rho)$ to be $\mathcal{Z}^k(\g;\rho)\!/\!\mathcal{B}^k(\g;\rho)$.

Let $(\ad,\g,\phi_\g)$ be the adjoint representation. For any 0-Hom-cochain $x\in \g=C^0(\g,\g)$, we have
\[(dx)(y)=[y,x]_{\mathfrak{g}},\qquad\forall\, y\in \mathfrak{g}.\] Thus, we have $dx=0$ if and only if $x\in \cen(\g)$, where $\cen(\g)$ denotes the center of $\g$.
 Therefore,
\[\mathcal{H}^{0}(\mathfrak{g},\ad)=\mathcal{Z}^{0}(\mathfrak{g},\ad)=\cen(\g).\]

\begin{defi}\label{def:Derivation}
A linear map $D\colon \mathfrak{g}\lon \mathfrak{g}$ is called a {\it derivation} of a Hom-Lie algebra $(\mathfrak{g},[\cdot,\cdot]_{\g}, \phi_{\g})$ if the following identity holds:
\begin{gather*}
\label{eq:37}D[x,y]_{\g}=\big[\phi(x),\big(\Ad_{\phi_{\g}^{-1}}D\big)(y)\big]_{\g}+\big[\big(\Ad_{\phi_{\g}^{-1}}D\big)(x),
\phi(y)\big]_{\g},\qquad\forall\, x,y\in\mathfrak{g}.
\end{gather*}
We denote the space of derivations of the Hom-Lie algebra $(\mathfrak{g},[\cdot,\cdot]_{\g}, \phi_{\g})$ by $\Der(\g)$.
\end{defi}

Let us observe that if $D\circ \phi=\phi\circ D$ in the Definition~\ref{def:Derivation}, then any derivation of the Hom-Lie algebra $(\mathfrak{g},[\cdot,\cdot]_{\g}, \phi_{\g})$ is a $\phi_{\g}$-derivation (see \cite{sheng} for more details).
\begin{ex}Let $(\g,[\cdot,\cdot]_{\g},\phi_{\g})$ be a Hom-Lie algebra. For each $x\in\g$, $\ad_x$ is a derivation of $(\mathfrak{g},[\cdot,\cdot]_{\g}, \phi_{\g})$ that we call an ``inner derivation''.
\end{ex}

 Let us denote the space of inner derivations by $\sf{InnDer}(\g)$. It is immediate to see that
\[\mathcal{Z}^{1}(\mathfrak{g},\ad)=\Der(\g),\qquad \mathcal{B}^{1}(\mathfrak{g},\ad)=\sf{InnDer}(\g).\]
Therefore,
\[ \mathcal{H}^{1}(\mathfrak{g},\ad)= \Der(\g)/\sf{InnDer}(\g)=\sf{OutDer}(\g).\]
Here, $\sf{OutDer}(\g)$ denotes the space of outer derivations of the Hom-Lie algebra $(\g,[\cdot,\cdot]_{\g},\phi_{\g})$.

Let $V$ be a vector space, and $\beta\in {\rm GL}(V)$. Let us define a skew-symmetric bilinear bracket operation $[\cdot,\cdot]_{\beta}\colon \wedge^2\mathfrak{gl}(V)\longrightarrow\mathfrak{gl}(V)$ by
\begin{gather*}\label{eq:bracket}
[A,B]_{\beta}=\beta \circ A \circ\beta^{-1}\circ B \circ\beta^{-1}-\beta\circ B \circ\beta^{-1}\circ A\circ \beta^{-1}, \qquad\forall\, A,B\in \mathfrak{gl}(V).
\end{gather*}
We also define a map $\Ad_{\beta}\colon \mathfrak{gl}(V)\lon \mathfrak{gl}(V)$ by
\begin{equation*}\label{eq:Ad}
\Ad_{\beta}(A)=\beta\circ A\circ \beta^{-1},\qquad\forall\, A\in \mathfrak{gl}(V).
\end{equation*}
With the above notations, we have the following proposition.
\begin{pro}[{\cite[Proposition 4.1]{shengxiong}}]\label{pro:Hom-Lie}
 The triple $(\mathfrak{gl}(V),[\cdot,\cdot]_{\beta},\Ad_{\beta})$ is a regular Hom-Lie algebra.
 \end{pro}

 The Hom-Lie algebra $(\mathfrak{gl}(V),[\cdot,\cdot]_{\beta},\Ad_{\beta})$ plays an important role in the representation theory of Hom-Lie algebras. See~\cite{shengxiong} for more details.

 \subsection{Hom-groups}
Throughout this paper, we consider regular Hom-groups that is the case when the structure map is invertible, and this notion can be traced back to Caenepeel and Goyvaerts's pioneering work~\cite{CatDes}. The axioms in the following definition of Hom-group is different from the one in~\cite{HG1,HG2,LGMT}. However, we show that if the structure map is invertible, then some axioms in the original definition are redundant and can be obtained from the Hom-associativity condition.

\begin{defi}\label{Hom-group}\samepage
A (regular) Hom-group is a set $G$ equipped with a product $\diamond\colon G\times G\longrightarrow G$, a~bijective map $\Phi\colon G\longrightarrow G$ such that the following axioms are satisfied
\begin{itemize}\itemsep=0pt
\item[\rm(i)] $\Phi(x\diamond y)=\Phi(x)\diamond \Phi(y)$;
\item[\rm(ii)] the product is Hom-associative, i.e.,
 \begin{equation*}
 \label{eq:morcon0}\Phi(x)\diamond (y\diamond z)= (x\diamond y)\diamond\Phi(z), \qquad\forall\, x,y, z\in G;
 \end{equation*}
 \item[\rm(iii)] there exists a unique Hom-unit $e_\Phi\in G$ such that
 \begin{equation*}
 x\diamond e_\Phi=e_\Phi\diamond x= \Phi (x),\qquad \forall\, x\in G;
 \end{equation*}

 \item[\rm(iv)] for each $x\in G$, there exists an element $x^{-1}\in G$ satisfying the following condition
 \begin{equation*}
 x\diamond x^{-1}=x^{-1}\diamond x=e_\Phi.
 \end{equation*}
 \end{itemize}
We denote a Hom-group by $(G,\diamond,e_\Phi,\Phi)$.
 \end{defi}

\begin{rmk}
The category of sets $(\textbf{Sets}, \times, \{\ast\}, \tau)$ (where $\tau$ is the twist) is a symmetric (strict) monoidal category. Algebras (monoids) in $\textbf{Sets}$ are also bialgebras since every set $X$ has a unique structure of coalgebra in $\textbf{Sets}$, namely $\Delta(x)=(x,x)$ and $\varepsilon(x)=\ast$, for all $x\in X$. A Hopf algebra in $\textbf{Sets}$ is a group. Let $X$ be a set and $\pi$ be a permutation. Then $(X, \pi, \Delta, \varepsilon)$ is a Hom-comonoid, where $\Delta\colon X\rightarrow X\times X$ and $ \varepsilon\colon X\rightarrow X$ are the maps given by $\Delta(x)=\big(\pi^{-1}(x),\pi^{-1}(x)\big)$ and $ \varepsilon(x)=\ast$. Similarly, if $\phi$ is an automorphism of a group $G$, then $(G,\phi)$ with structure maps
\begin{gather*}
g\cdot h=\phi(gh),\qquad \eta(\ast)=1_{G}, \qquad \varepsilon(g)=\ast,\qquad
\Delta(g)=\big(\phi^{-1}(g), \phi^{-1}(g)\big), \qquad S(g)=g^{-1},
\end{gather*}
is a Hom-group, that is a Hopf algebra in $\tilde{\huaH}(\textbf{Sets})$ in \cite[Section~5]{CatDes}. Thus a monoidal categorical
approach can give an intrinsic study of regular Hom-groups.
\end{rmk}

\begin{rmk}
Note that the definition of a Hom-group in~\cite{HG2} consists of the axiom $\Phi(e_{\Phi})=e_{\Phi}$. In Proposition~\ref{pro:0}, we show that this axiom is redundant in the regular case. Let us recall the Hom-invertibility condition in the definition of a Hom-group $(G,\Phi)$ in~\cite{LGMT}: for each $x\in G$, there exists a positive integer~$k$ such that
\[\Phi^{k}\big(x\diamond x^{-1}\big)=\Phi^k\big(x^{-1}\diamond x\big)=e_{\Phi},\]
and the smallest such integer $k$ is called the invertibility index of $x\in G$. In the regular case, it is immediate to see that the Hom-invertibility condition is equivalent to the condition~(iv) in Definition~\ref{Hom-group}.
\end{rmk}

\begin{ex}Let $(G,\mu,e)$ be a group and $\phi\colon G\rightarrow G$ be a group automorphism. Then the tuple $(G,\mu_{\phi},e,\phi)$ with the product $\mu_{\phi}=\phi\circ\mu,$ is a Hom-group. In particular, the tuple $(\mathbb{R},+, 0,\Id )$ is a Hom-group, which will be used in our later definition of one-parameter Hom-Lie subgroups.
\end{ex}

It is straightforward to obtain the following properties, which were also given in~\cite{CatDes}.
\begin{pro}\label{pro:0}
Let $(G,\diamond,e_\Phi,\Phi)$ be a Hom-group. Then we have the following properties.
\begin{itemize}\itemsep=0pt
\item [$(i)$] $\Phi(e_{\Phi})=e_{\Phi};$
\item [$(ii)$] for each $x\in G$, there exists a unique inverse $x^{-1}_\Phi$ such that
 \begin{equation*}
 x\diamond x^{-1}_\Phi=x^{-1}_\Phi\diamond x=e_\Phi;
 \end{equation*}
\item [$(iii)$] $(x\diamond y)^{-1}=y^{-1}\diamond x^{-1}$, $\forall\, x,y\in G$.
\end{itemize}
\end{pro}

\begin{defi}
Let $(G,\diamond_G,e_\Phi,\Phi)$ and $(H,\diamond_H,e_\Psi,\Psi)$ be two Hom-groups.
\begin{enumerate}\itemsep=0pt
\item[(a)] A {\it homomorphism} of Hom-groups is a map $f\colon G\lon H$ such that
 $f(e_\Phi)=e_\Psi$ and $f(x\diamond_Gy)=f(x)\diamond_H f(y)$ for all $x, y\in G$.

\item[(b)] A {\it weak homomorphism} of Hom-groups is a map $f\colon G\lon H$ such that
 $f(e_\Phi)=e_\Psi$ and
 $\Psi \circ f(x\diamond_G y)=\big(f\circ \Phi(x)\big)\diamond_H \big(f\circ \Phi(y)\big)$ for all $x,y\in G$.
\end{enumerate}
\end{defi}
Let us observe that for a homomorphism $f\colon (G,\diamond_G,e_\Phi,\Phi)\lon (H,\diamond_H,e_\Psi,\Psi)$, the commutativity condition: $\Psi\circ f=f\circ \Phi$ holds. It follows by the definition of a homomorphism and the identities: $\Phi(e_{\Phi})=e_{\Phi},$ and $\Psi(e_{\Psi})=e_{\Psi}$. Furthermore, any homomorphism of Hom-groups is also a weak homomorphism, however, the converse may not be true.

\section{Hom-Lie groups and Hom-Lie algebras}\label{Homliealgebra}

Let $\mathbb{R}$ be the field of real numbers. From here onwards, we consider all manifolds, vector spaces over the field $\mathbb{R}$, and all the linear maps are considered to be $\mathbb{R}$-linear unless otherwise stated.
\subsection{Hom-Lie groups}

\begin{defi}\label{def:HL-group}
A real Hom-Lie group is a Hom-group $(G,\diamond,e_{\Phi},\Phi)$, in which $G$ is also a smooth real manifold, the map $\Phi\colon G\rightarrow G$ is a diffeomorphism, and the Hom-group operations (product and inversion) are smooth maps with respect to the topology of~$G$.
\end{defi}

\begin{ex}\label{rmk:00}Let $(G,\cdot)$ be a Lie group with identity $e$ and $\Phi\colon (G,\cdot)\rightarrow (G,\cdot)$ be an automorphism. Then the tuple $(G,\diamond,e_\Phi=e,\Phi)$ is a Hom-Lie group, where the product $\diamond$ is defined by
\[ a\diamond b= \Phi(a)\cdot \Phi(b),\qquad\forall\, a, b \in G.\]
\end{ex}

Let $V$ be a vector space and $\beta\in {\rm GL}(V)$. Then let us define a product $\diamond\colon {\rm GL}(V)\times {\rm GL}(V)\longrightarrow {\rm GL}(V)$ by
\begin{eqnarray}\label{eq:mu1}
 A\diamond B =\beta \circ A \circ\beta^{-1}\circ B \circ\beta^{-1},\qquad \forall\, A,B\in {\rm GL}(V).
\end{eqnarray}

\begin{pro}\label{HGrp:GL(V)}
 The tuple $({\rm GL}(V),\diamond,\beta,\Ad_{\beta})$ is a Hom-Lie group, where the product $\diamond$ is given by \eqref{eq:mu1}, the Hom-unit is $\beta$, and the map $\Ad_{\beta}\colon {\rm GL}(V)\rightarrow {\rm GL}(V)$ is defined by
 \[\Ad_\beta(A)=\beta\circ A\circ \beta^{-1},\qquad \forall\, A\in {\rm GL}(V).\]
\end{pro}

\begin{proof} For all $A, B\in {\rm GL}(V)$, we have
\begin{align*}
\Ad_{\beta}( A\diamond B)&= \Ad_{\beta}\big(\beta\circ A\circ\beta^{-1}\circ B \circ\beta^{-1}\big)
 = \beta^2\circ A\circ\beta^{-1}\circ B \circ\beta^{-2}\\
 &= \beta\circ \Ad_{\beta}(A)\circ\beta^{-1}\circ \Ad_{\beta}(B) \circ\beta^{-1}
 = \Ad_{\beta}(A)\diamond \Ad_{\beta}(B).
\end{align*}
Thus, condition (i) in Definition \ref{Hom-group} holds.
For all $A, B, C\in {\rm GL}(V)$, it easily follows that
\begin{align*}
(A\diamond B)\diamond(\Ad_{\beta}(C))&= \big(\beta \circ A \circ\beta^{-1}\circ B \circ\beta^{-1}\big)\diamond\big(\beta\circ C\circ \beta^{-1}\big)\\
 &= \beta^2\circ A \circ\beta^{-1}\circ B \circ\beta^{-1}\circ C\circ \beta^{-2}\\
 & =\beta\circ\Ad_{\beta}(A)\circ\beta^{-1}\circ(B\diamond C)\circ\beta^{-1}\\
 &=\Ad_{\beta}(A)\diamond(B\diamond C),
\end{align*}
which implies that the product $\diamond$ is Hom-associative. Next, we have
\begin{gather*}
A\diamond \beta=\beta\circ A\circ\beta^{-1}\circ\beta\circ\beta^{-1}=\Ad_{\beta}(A),\qquad\forall\, A\in {\rm GL}(V).
\end{gather*}
Similarly,
\[\beta\diamond A=\Ad_{\beta}(A),\qquad\forall\, A\in {\rm GL}(V).\]
Therefore, $\beta$ is the Hom-unit. Finally, we have the following expression
\begin{gather*}
A\diamond \big(\beta\circ A^{-1}\circ \beta\big) = \beta \circ A \circ\beta^{-1}\circ \big(\beta\circ A^{-1}\circ \beta\big) \circ\beta^{-1}=\beta=\big(\beta\circ A^{-1}\circ \beta\big)\diamond A,
\end{gather*}
for any $A\in {\rm GL}(V)$, i.e., $\beta\circ A^{-1}\circ \beta$ is the Hom-inverse of $A$. Hence, the tuple $({\rm GL}(V),\diamond,\beta,\Ad_{\beta})$ is a Hom-group.
\end{proof}

\begin{ex}
Let $M$ be a smooth manifold. Let us denote by $\Diff(M)$, the set of diffeomorphisms of $M$. If $\iota\in \Diff(M)$, then the tuple $(\Diff(M),\diamond,\iota,\Ad_{\iota})$ is a Hom-Lie group, where
\begin{itemize}\itemsep=0pt
\item [(i)] the product $\diamond$ is given by the following equation$:$
\[ f\diamond g=\iota\circ f\circ\iota^{-1}\circ g\circ\iota^{-1},\qquad\forall\, f, g\in \Diff(M);\]
\item [(ii)] the Hom-unit is $\iota\in \Diff(M)$;
\item [(iii)] the map $\Ad_{\iota}\colon \Diff(M)\rightarrow \Diff(M)$ is defined by $\Ad_{\iota}(f)=\iota\circ f\circ\iota^{-1}$, for all $~f\in \Diff(M)$.
\end{itemize}
\end{ex}

Let $(G,\diamond,e_\Phi,\Phi)$ be a Hom-Lie group and $TG$ be the tangent bundle of the manifold $G$. Let us denote by $\Phi^{!}TG$, the pullback bundle of the tangent bundle $TG$ along the diffeomorphism $\Phi\colon G\rightarrow G$. Then we have
the following one-to-one correspondence.

\begin{lem}\label{lem:i1}
There is a one-to-one correspondence between the space of sections of the tangent bundle $TG$ $(\mbox{i.e.}, \Gamma(TG))$ and the space of sections of the pullback bundle $\Phi^{!}TG$ $\big($i.e., $\Gamma \big(\Phi^{!}TG\big)\big)$.
\end{lem}

\begin{proof} Let $X\in \Gamma(TG)$, then define a smooth map $x\colon G\rightarrow TG$ by $x=X\circ\Phi$. Let us consider the set $\Gamma_{\Phi^{!}}(TG)=\{x\colon G\rightarrow TG\,|\, x=X\circ\Phi\}$. Since the map $\Phi\colon G\rightarrow G$ is a diffeomorphism, there is a one-to-one correspondence between the sets $\Gamma(TG)$ and $\Gamma_{\Phi^{!}}(TG)$:
\[
 \xymatrix{
 \Phi^{!}TG \ar[r]^{{\rm pr}}
 & TG \\
 G \ar[u]_{X^{!}} \ar@{.>}[ur]|-{x} \ar[r]_{\Phi}
 & G. \ar[u]^{X} }
\]
Note that there exists a unique $X^{!}$ such that $x={\rm pr}\circ X^{!}$. Hence, there is a one-to-one correspondence between $\Gamma (TG)$ with $\Gamma \big(\Phi^{!}TG\big)$.
\end{proof}

For $X\in \Gamma (TG)$, we denote the corresponding pullback section by $X^{!}\in \Gamma \big(\Phi^{!}TG\big)$. Through this paper, we identify $X^{!}\in \Gamma \big(\Phi^{!}TG\big)$ by $x\in \Gamma_{\Phi^{!}}(TG)$. Let us observe that if we define
\begin{equation}
 \label{eq:impor1}x(f)=X(f)\circ\Phi, \qquad \forall\, f \in C^{\infty}(G),
 \end{equation}
 then $\Gamma_{\Phi^{!}}(TG)$ can be identified with the set of $(\Phi^{*}, \Phi^{*})$-derivations $\Der_{\Phi^{*}, \Phi^{*}}(C^{\infty}(G))$ on $C^{\infty}(G)$, i.e., for all $f,g\in C^{\infty}(G), x\in \Gamma_{\Phi^{!}}(TG)$, we have
\begin{equation}
\label{eq:morcon3}x(fg)=x(f)\Phi^{*}(g)+\Phi^{*}(f)x(g).
\end{equation}
Thus, the space of sections $\Gamma \big(\Phi^{!}TG\big)$ can be identified with the space of $(\Phi^{*}, \Phi^{*})$-derivations of on $C^{\infty}(G)$, i.e., $\Der_{\Phi^{*}, \Phi^{*}}(C^{\infty}(G))$. In the following theorem, we define a Hom-Lie algebra structure on the space of sections $\Gamma \big(\Phi^{!}TG\big)$.

\begin{thm}
Let $G$ be a smooth manifold. Then $\big(\Gamma \big(\Phi^{!}TG\big), [\cdot ,\cdot]_{\Phi} , \phi\big)$ is a Hom-Lie algebra, where the Hom-Lie bracket $[\cdot ,\cdot]_{\Phi}$ and the map $\phi\colon \Gamma \big(\Phi^{!}TG\big)\rightarrow \Gamma \big(\Phi^{!}TG\big)$ are defined as follows:
\begin{gather}
 \label{eq:morcon4}\phi(x)=\big(\Phi^{-1}\big)^{*}\circ x\circ \Phi^{*},\\
 \label{eq:morcon5}[x, y]_{\Phi}=\big(\Phi^{-1}\big)^{*}\circ x \circ\big(\Phi^{-1}\big)^{*} \circ y \circ \Phi^{*}-\big(\Phi^{-1}\big)^{*}\circ y\circ \big(\Phi^{-1}\big)^{*}\circ x\circ \Phi^{*},
\end{gather}
for any $x,y\in\Gamma \big(\Phi^{!}TG\big)$.
\end{thm}

\begin{proof} Let $x, y \in \Gamma \big(\Phi^{!}TG\big)$ and $f,g \in C^{\infty}(G)$. Then by (\ref{eq:morcon3}) and (\ref{eq:morcon4}), we get
\begin{align*}
\phi(x)(fg)&=\big(\Phi^{-1}\big)^{*}\circ x\circ\Phi^{*}(fg)\\
&=\big(x(f\circ\Phi)\circ\Phi^{-1}\big)(g\circ\Phi)+(f\circ\Phi)\big(x(g\circ\Phi)\circ\Phi^{-1}\big)\\
&=\phi(x)(f)(g\circ\Phi)+(f\circ\Phi)\phi(x)(g)\\
& =\phi(x)(f)\Phi^{*}(g)+\Phi^{*}(f)\phi(x)(g),
\end{align*}
which implies that $\phi(x)$ is a $(\Phi^{*}, \Phi^{*})$-derivation on $C^{\infty}(G)$ and hence, $\phi(x) \in \Gamma \big(\Phi^{!}TG\big)$. Next, for any $x, y \in \Gamma \big(\Phi^{!}TG\big)$ and $f,g \in C^{\infty}(G)$, we have
\begin{gather*}
[x, y]_{\Phi}(fg)= \big(\Phi^{-1}\big)^{*}\circ x\circ\big(\Phi^{-1}\big)^{*}\circ y\circ \Phi^{*}(fg)-\big(\Phi^{-1}\big)^{*}\circ y\circ\big(\Phi^{-1}\big)^{*}\circ x\circ\Phi^{*}(fg)\\
\hphantom{[x, y]_{\Phi}(fg)}{} =x\big(y((f\circ\Phi)(g\circ\Phi))\circ\Phi^{-1}\big)\circ\Phi^{-1}-y\big(x((f\circ\Phi)(g\circ\Phi))\circ\Phi^{-1}\big)\circ\Phi^{-1}\\
\hphantom{[x, y]_{\Phi}(fg)}{} =x\big((y(f\circ\Phi)\big(g\circ\Phi^{2}\big)+\big(f\circ \Phi^{2}\big)y(g\circ\Phi))\circ\Phi^{-1}\big)\circ\Phi^{-1}\\
\hphantom{[x, y]_{\Phi}(fg)=}{} -y\big((x(f\circ\Phi)\big(g\circ\Phi^{2}\big)+\big(f\circ\Phi^{2}\big)x\big(g\circ\Phi^{-1}\big)\big)\circ\Phi^{-1})\circ\Phi^{-1}\\
\hphantom{[x, y]_{\Phi}(fg)}{} =\big(x\big(y(f\circ\Phi)\circ\Phi^{-1}\big)\circ\Phi^{-1}\big)(g\circ\Phi)-\big(y\big(x(f\circ\Phi)\circ\Phi^{-1}\big) \circ\Phi^{-1}\big)(g\circ\Phi)\\
\hphantom{[x, y]_{\Phi}(fg)=}{} +(f\circ\Phi)\big(x\big(y(g\circ\Phi)\circ\Phi^{-1}\big)\circ\Phi^{-1}\big) -(f\circ\Phi)\big(y\big(x(g\circ\Phi)\circ\Phi^{-1}\big)\circ\Phi^{-1}\big),
\end{gather*}
and
\begin{gather*}
 [x, y]_{\Phi}(f)\Phi^{*}(g)+\Phi^{*}(f)[x, y]_{\Phi}(g)\\
\qquad{} = \big(x\big(y(f\circ\Phi)\circ\Phi^{-1}\big)\circ\Phi^{-1}\big)(g\circ\Phi) -\big(y\big(x(f\circ\Phi)\circ\Phi^{-1}\big)\circ\Phi\big)(g\circ\Phi)\\
\qquad\quad {} +(f\Phi)\big(x\big(y(g\Phi^{-1})\Phi^{-1}\big)\Phi\big)-(f\Phi)\big(y\big(x(g\Phi^{-1})\Phi^{-1}\big)\Phi\big).
\end{gather*}
i.e.,
\begin{equation*}
[x, y]_{\Phi}(fg)=[x, y]_{\Phi}(f)\Phi^{*}(g)+\Phi^{*}(f)[x, y]_{\Phi}(g),
\end{equation*}
which implies that $[x, y]_{\Phi} \in \Gamma \big(\Phi^{!}TG\big)$.
Moreover, by~\eqref{eq:morcon4} and~\eqref{eq:morcon5}, we get the following expressions:
\[ \phi[x, y]_{\Phi}=\big(\Phi^{-2}\big)^{*}\circ x\circ\big(\Phi^{-1}\big)^{*}\circ y\circ\big(\Phi^{2}\big)^{*}-\big(\Phi^{-2}\big)^{*}\circ y\circ\big(\Phi^{-1}\big)^{*}\circ x\circ\big(\Phi^{2}\big)^{*},\]
and
\[ [\phi(x), \phi(y)]_{\Phi}=\big(\Phi^{-2}\big)^{*}\circ x\circ\big(\Phi^{-1}\big)^{*}\circ y\circ\big(\Phi^{2}\big)^{*}-\big(\Phi^{-2}\big)^{*}\circ y\circ\big(\Phi^{-1}\big)^{*}\circ x\circ\big(\Phi^{2}\big)^{*},\]
 which, in turn, implies that $\phi[x, y]_{\Phi}=[\phi(x), \phi(y)]_{\Phi}$.

Finally, we have
\begin{gather*}
[\phi(x),[y,z]_{\Phi}]_{\Phi}= \big(\Phi^{-2}\big)^{*}\circ x\circ\big(\Phi^{-1}\big)^{*}\circ y\circ\big(\Phi^{-1}\big)^{*}\circ z\circ\big(\Phi^{2}\big)^{*}\\
\hphantom{[\phi(x),[y,z]_{\Phi}]_{\Phi}=}{}
-\big(\Phi^{-2}\big)^{*}\circ x\circ\big(\Phi^{-1}\big)^{*}\circ z\circ\big(\Phi^{-1}\big)^{*}\circ y\circ\big(\Phi^{2}\big)^{*}\\
\hphantom{[\phi(x),[y,z]_{\Phi}]_{\Phi}=}{}
-\big(\Phi^{-2}\big)^{*}\circ y\circ\big(\Phi^{-1}\big)^{*}\circ z\circ\big(\Phi^{-1}\big)^{*}\circ x\circ\big(\Phi^{2}\big)^{*}\\
\hphantom{[\phi(x),[y,z]_{\Phi}]_{\Phi}=}{}
+\big(\Phi^{-2}\big)^{*}\circ z\circ\big(\Phi^{-1}\big)^{*}\circ y\circ\big(\Phi^{-1}\big)^{*}\circ x\circ\big(\Phi^{2}\big)^{*}.
\end{gather*}
Similarly, we have
\begin{gather*}
 [\phi(y),[z,x]_{\Phi}]_{\Phi}= \big(\Phi^{-2}\big)^{*}\circ y\circ\big(\Phi^{-1}\big)^{*}\circ z\circ\big(\Phi^{-1}\big)^{*}\circ x\circ\big(\Phi^{2}\big)^{*}\\
 \hphantom{[\phi(y),[z,x]_{\Phi}]_{\Phi}=}{}
 -\big(\Phi^{-2}\big)^{*}\circ y\circ\big(\Phi^{-1}\big)^{*}\circ x\circ\big(\Phi^{-1}\big)^{*}\circ z\circ\big(\Phi^{2}\big)^{*}\\
\hphantom{[\phi(y),[z,x]_{\Phi}]_{\Phi}=}{}
-\big(\Phi^{-2}\big)^{*}\circ z\circ\big(\Phi^{-1}\big)^{*}\circ x\circ\big(\Phi^{-1}\big)^{*}\circ y\circ\big(\Phi^{2}\big)^{*}\\
\hphantom{[\phi(y),[z,x]_{\Phi}]_{\Phi}=}{}
+\big(\Phi^{-2}\big)^{*}\circ x\circ\big(\Phi^{-1}\big)^{*}\circ z\circ\big(\Phi^{-1}\big)^{*}\circ y\circ\big(\Phi^{2}\big)^{*},
\end{gather*}
and
\begin{gather*}
 [\phi(z),[x,y]_{\Phi}]_{\Phi} =\big(\Phi^{-2}\big)^{*}\circ z\circ\big(\Phi^{-1}\big)^{*}\circ x\circ\big(\Phi^{-1}\big)^{*}\circ y\circ\big(\Phi^{2}\big)^{*}\\
 \hphantom{[\phi(z),[x,y]_{\Phi}]_{\Phi} =}{}
 -\big(\Phi^{-2}\big)^{*}\circ z\circ\big(\Phi^{-1}\big)^{*}\circ y\circ\big(\Phi^{-1}\big)^{*}\circ x\circ\big(\Phi^{2}\big)^{*}\\
\hphantom{[\phi(z),[x,y]_{\Phi}]_{\Phi} =}{}
-\big(\Phi^{-2}\big)^{*}\circ y\circ\big(\Phi^{-1}\big)^{*}\circ z\circ\big(\Phi^{-1}\big)^{*}\circ z\circ\big(\Phi^{2}\big)^{*}\\
\hphantom{[\phi(z),[x,y]_{\Phi}]_{\Phi} =}{}
+\big(\Phi^{-2}\big)^{*}\circ y\circ\big(\Phi^{-1}\big)^{*}\circ x\circ\big(\Phi^{-1}\big)^{*}\circ z\circ\big(\Phi^{2}\big)^{*},
\end{gather*}
which implies that
\begin{equation*}
[\phi(x),[y,z]_{\Phi}]_{\Phi}+[\phi(y),[z,x]_{\Phi}]_{\Phi}+[\phi(z),[x,y]_{\Phi}]_{\Phi}=0.
\end{equation*}
Therefore, $\big(\Gamma\big(\Phi^{!}TG\big),[\cdot,\cdot]_\Phi,\phi\big)$ is a Hom-Lie algebra.
\end{proof}

\subsection{The Hom-Lie algebra of a Hom-Lie group}
Let $(G,\diamond, e_\Phi,\Phi)$ be a Hom-Lie group. For $a\in G$, let us define a smooth map
\[ l_{a} \colon \ G\rightarrow G \qquad \mbox{by}\quad l_{a}(b)=a\diamond b,\qquad\forall\, b \in G.
\]
Then the smooth map $l_{a}\colon G\rightarrow G$ is a diffeomorphism (by Definition~\ref{def:HL-group}).

\begin{defi}
Let $(G, \diamond, e_\Phi, \Phi)$ be a Hom-Lie group. A smooth section $x\in \Gamma \big(\Phi^{!}TG\big)$ is left-invariant if $x$ satisfies the following equation$:$
\begin{equation}
\label{eq:morcon6}x(f)(a)=x\big(f\circ l_{a}\circ\Phi^{-1}\big)(e_{\Phi}),\qquad\forall\, a \in G, f \in C^{\infty}(G).
\end{equation}
 \end{defi}

Let us denote by $\Gamma_{L}\big(\Phi^{!}TG\big)$, the space of all left-invariant sections of the pullback bundle $\Phi^{!}TG$. Next, we show that the space $\Gamma_{L}\big(\Phi^{!}TG\big)$ carries a Hom-Lie algebra structure. In fact, we prove that $\big(\Gamma_{L}\big(\Phi^{!}TG\big),[\cdot,\cdot]_{\Phi},\phi\big)$ is a Hom-Lie subalgebra of the Hom-Lie algebra $\big(\Gamma\big(\Phi^{!}TG\big),[\cdot,\cdot]_{\Phi},\phi\big)$.

\begin{lem} \label{lem:imp1}
Let $(G, \diamond, e_\Phi, \Phi)$ be a Hom-Lie group and $x\in \Gamma\big(\Phi^{!}TG\big)$ be a left-invariant section. Then we have
\begin{equation}
\label{eq:morcon7}x(f\circ\Phi)\circ\Phi^{-1}\circ l_{\Phi^{-1}(a)}=x(f\circ l_{a}),\qquad\forall\, a \in G.
\end{equation}
\end{lem}

\begin{proof} First, let us note that by using the Hom-associativity condition of the product $\diamond$, we get the following equation:
\begin{equation*}
\label{eq:morcon8}\Phi\circ l_{(\Phi^{-2}(a)\diamond\Phi^{-1}(b))}=l_{a}\circ l_{b},\qquad\forall\, a,b\in G,
\end{equation*}
which implies that
\begin{equation}
\label{eq:w1}f\circ\Phi\circ l_{(\Phi^{-2}(a)\diamond\Phi^{-1}(b))}\circ\Phi^{-1}=f\circ l_{a}\circ l_{b}\circ\Phi^{-1}.
\end{equation}

By using the left-invariant property \eqref{eq:morcon6} of the section $x$, we have
\begin{equation}\label{eq1}
x(f\circ\Phi)\circ\Phi^{-1}\circ l_{\Phi^{-1}(a)}(b)=x\big(f\circ\Phi\circ l_{(\Phi^{-2}(a)\diamond\Phi^{-1}(b))}\circ\Phi^{-1}\big)(e_{\Phi}),
\end{equation}
and
\begin{equation}\label{eq2}
x(f\circ l_{a})(b)=x\big(f\circ l_{a}\circ l_{b}\circ\Phi^{-1}\big)(e_{\Phi}).
\end{equation}
Thus, by \eqref{eq:w1}--\eqref{eq2}, we deduce that the desired identity~(\ref{eq:morcon7}) holds.
\end{proof}

\begin{thm}
The space $\Gamma_{L}\big(\Phi^{!}TG\big)$ of left-invariant sections of the pullback bundle $\Phi^{!}TG$ is a Hom-Lie subalgebra of the Hom-Lie algebra $\big(\Gamma\big(\Phi^{!}TG\big), [\cdot, \cdot]_{\Phi}, \phi\big)$.
\end{thm}

\begin{proof} First, let us prove that $\phi(x) \in \Gamma_{L}\big(\Phi^{!}TG\big)$ for any $x\in\Gamma_{L}\big(\Phi^{!}TG\big)$. By \eqref{eq:morcon4} and~\eqref{eq:morcon6}, we have
\begin{gather*}
\phi(x)(f)(a) =\big(\Phi^{-1}\big)^{*}\circ x\circ\Phi^{*}(f)(a)
 = x(f\circ\Phi)\big(\Phi^{-1}(a)\big)\\
\hphantom{\phi(x)(f)(a)}{} = x\big(f\circ\Phi\circ l_{\Phi^{-1}(a)}\Phi^{-1}\big)(e)
 = x(f\circ l_{a})(e) = \phi(x)\big(f\circ l_{a}\circ\Phi^{-1}\big)(e)
\end{gather*}
for all $x, y \in \Gamma_{L}\big(\Phi^{!}TG\big)$, and $a\in G$. This, in turn, implies that $\phi(x)\in \Gamma_{L}\big(\Phi^{!}TG\big)$.

Now we prove that $[x, y]_{\Phi} \in \Gamma_{L}\big(\Phi^{!}TG\big)$. By~\eqref{eq:morcon5} and~\eqref{eq:morcon6}, we have the following expressions:
\begin{gather*}
[x, y]_{\Phi}(f)(a)=x\big(y(f\circ\Phi)\circ\Phi^{-1}\big)\big(\Phi^{-1}(a)\big)-y\big(x(f\circ\Phi)\circ\Phi^{-1}\big)\big(\Phi^{-1}(a)\big)\\
\hphantom{[x, y]_{\Phi}(f)(a)}{}
= x\big(y(f\circ\Phi)\circ\Phi^{-1}\circ l_{\Phi^{-1}(a)}\circ\Phi^{-1}\big)(e_{\Phi})\\
\hphantom{[x, y]_{\Phi}(f)(a)=}{} -y\big(x(f\circ\Phi)\circ\Phi^{-1}\circ l_{\Phi^{-1}(a)}\circ\Phi^{-1}\big)(e_{\Phi}),
\end{gather*}
and
\[ [x,y]_{\Phi}\big(f\circ l_{a}\circ\Phi^{-1}\big)(e_{\Phi})=x\big(y(f\circ l_{a})\circ\Phi^{-1}\big)(e_{\Phi})-y\big(x(f\circ l_{a})\circ\Phi^{-1}\big)(e_{\Phi})\]
for all $x, y \in \Gamma_{L}\big(\Phi^{!}TG\big)$ and $a\in G$. Thus, from Lemma~\ref{lem:imp1}, we have
\[ [x, y]_{\Phi}(f)(a)=[x,y]_{\Phi}\big(f\circ l_{a}\circ\Phi^{-1}\big)(e_{\Phi}),\] which implies that $[x,y]_{\Phi} \in \Gamma_{L}\big(\Phi^{!}TG\big)$. The proof is finished.
\end{proof}

\begin{rmk}\label{Rem:LIE-HLIE}
Let $(G,\diamond,e_\Phi,\Phi)$ be a Hom-Lie group. Then we get a Lie group structure $(G,\cdot,e_\Phi)$ equipped with the product $\cdot\colon G\times G\rightarrow G$ defined by $a\cdot b=\Phi^{-1}(a\diamond b)$ for all $a,b\in G$.
\end{rmk}

\begin{lem}\label{lem:important2}
Let $(G,\diamond,e_\Phi,\Phi)$ be a Hom-Lie group. Let $x$ be a section of $\Phi^{!}TG$ and $X$ be the corresponding section of~$TG$. Then $x$ is left-invariant if and only if $X$ is a left-invariant vector field of the associated Lie group $(G,\cdot,e_\Phi)$ $($by Remark~{\rm \ref{Rem:LIE-HLIE})}.
\end{lem}

\begin{proof} If $x \in \Gamma_{L}\big(\Phi^{!}TG\big)$, then by the definition of a left-invariant section, we get
\[ x(f)(a)=x\big(f\circ l_{a}\circ\Phi^{-1}\big)(e_\Phi),\qquad\forall\, f\in C^{\infty}(G), \qquad a\in G.\]
Let $X$ be the corresponding section of~$TG$, i.e., $x=X\circ\Phi$. Then we obtain the following expression:
\begin{equation*}
X(f)(\Phi(a))=X\big(f\circ l_{a}\circ\Phi^{-1}\big)(e_\Phi)=X(f\circ L_{\Phi(a)})(e_\Phi),
\end{equation*}
where $L_{\Phi(a)}(b)=\Phi(a)\cdot b=a\diamond \Phi^{-1}(b)$. Thus, $X$ is a left invariant vector field of the Lie group $(G, \cdot, e_\Phi)$.

Similarly, if $X\in \Gamma(TG)$ is a left-invariant vector field of the Lie group $(G, \cdot, e_\Phi)$, then we can deduce that the corresponding section $x\in \Gamma\big(\Phi^{!}TG\big)$ is left-invariant. We omit the details.
\end{proof}

Let $(G, \diamond, e_\Phi, \Phi)$ be a Hom-Lie group. Let us denote by $\g^{!}$, the fibre of $e_{\Phi}$ in the pullback bundle $\Phi^{!}TG$. Notice that $\g=T_{e_{\Phi}}G=\Phi^{!}T_{e_{\Phi}}G=\g^{!}$ (since, $\Phi(e_{\Phi})=e_{\Phi}$). Then by Lemma~\ref{lem:important2}, $\g^{!}$ is in one-to-one correspondence with $\Gamma_{L}(\Phi^{!}TG)$. With this in mind, it is natural to define a~bracket $[\cdot,\cdot]_{\g^{!}}$ and a vector space isomorphism $\phi_{\g^{!}}\colon \g^{!}\rightarrow \g^{!}$ as follows:
\begin{gather*}
[x(e_\Phi), y(e_\Phi)]_{\g^{!}} = [x, y]_{\Phi}(e_\Phi),\\
\phi_{\g^{!}}(x(e_\Phi)) = (\phi(x))(e_\Phi),
\end{gather*}
for all $x, y \in \Gamma_{L}\big(\Phi^{!}TG\big)$. It follows that the triple $\big(\mathfrak{g}^{!},[\cdot,\cdot]_{\g^{!}},\phi_{\g^{!}}\big)$ is a Hom-Lie algebra and it is isomorphic to the Hom-Lie algebra $\big(\Gamma_{L}\big(\Phi^{!}TG\big),[\cdot, \cdot]_{\Phi},\phi\big)$.

\begin{lem}\label{lem:Des1}
Let $(G,\diamond,e_\Phi,\Phi)$ be a Hom-Lie group. If $x, y \in \Gamma_{L}\big(\Phi^{!}TG\big)$, and $X$, $Y$ are the corresponding left-invariant vector fields of the Lie group $(G,\cdot,e_\Phi)$, then we obtain the following identities:
\begin{gather*}
[x, y]_{\Phi}(e_\Phi)=\Phi_{*e_\Phi}([X,Y](e_\Phi)),\\
\phi(x)(e_\Phi)=\Phi_{*e_\Phi}(X(e_\Phi)).
\end{gather*}
Here, the map $\Phi_*\colon TG\rightarrow TG$ is the differential of the smooth map $\Phi\colon G\rightarrow G$.
\end{lem}
\begin{proof} By \eqref{eq:impor1} and \eqref{eq:morcon4}, we get
\begin{gather*}
[x, y]_{\Phi}(f) = x\big(y(f\circ\Phi)\circ\Phi^{-1}\big)\circ\Phi^{-1}-y\big(x(f\circ\Phi)\circ\Phi^{-1}\big)\circ\Phi^{-1}\\
 \hphantom{[x, y]_{\Phi}(f)}{} = x(Y(f\circ\Phi))\circ\Phi^{-1}-y(X(f\circ\Phi))\circ\Phi^{-1}\\
\hphantom{[x, y]_{\Phi}(f)}{} = X(Y(f\circ\Phi))-Y(X(f\circ\Phi)) = [X, Y](f\circ\Phi)
\end{gather*}
for all $f \in C^{\infty}(G)$. Let $Z\in\Gamma(TG)$ be the corresponding section of $[x, y]_{\Phi}\in\Gamma\big(\Phi^{!}TG\big)$, i.e., $Z\circ\Phi=[x, y]_{\Phi}$. Then, we get the following expressions:
\[ Z(f)=[x, y]_{\Phi}(f)\circ\Phi^{-1}=([X,Y](f\circ\Phi))\circ\Phi^{-1},\]
 and
\[ Z(f)(e_{\Phi})=[x, y]_{\Phi}(f)(e_{\Phi})=[X,Y](f\circ\Phi)(e_{\Phi})=\Phi_{*e_\Phi}([X,Y](e_\Phi))(f).
\]
 Thus, $Z(e_\Phi)=\Phi_{*e_\Phi}([X,Y](e_\Phi))$ and we deduce that $[x, y]_{\Phi}(e_\Phi)=\Phi_{*e_\Phi}([X,Y](e_\Phi))$.

 Next, let us assume that $W\in\Gamma(TG)$ is the corresponding section of $\phi(x)\in\Gamma\big(\Phi^{!}TG\big)$. Since $\phi(x)(e_\Phi)=W({e_\Phi})$, we have
 \[ W(f)(e_\Phi)=\phi(x)(f)(e_\Phi)=x(f \circ\Phi)(e_\Phi)=X(f \circ \Phi)(e_\Phi),\]
which implies that $\phi(x)(e_\Phi)=W_{e_\Phi}=\Phi_{*e_\Phi}(X(e_\Phi))$.
\end{proof}

At the end of this subsection, we show that every regular Hom-Lie algebra is integrable.
\begin{defi}
A Hom-Lie group $(G,\diamond,e_\Phi,\Phi)$ is called simply connected Hom-Lie group if the underlying manifold~$G$ is a simply connected topological space.
\end{defi}
\begin{thm}
Let $(\g,[\cdot,\cdot]_{\g},\phi_{\g})$ be a regular Hom-Lie algebra. Then there exists a unique simply connected Hom-Lie group $(G,\diamond,e_\Phi,\Phi)$ such that $\g=\g^{!}$ and $\phi_{\g}=\Phi_{*e_\Phi}=\phi_{\g^{!}}$, where $\big(\g^{!},[\cdot,\cdot]_{\g^{!}},\phi_{\g^{!}}\big)$ is the associated Hom-Lie algebra.
\end{thm}

\begin{proof}
For the Lie algebra $(\g,[\cdot,\cdot]_{\rm Lie})$ given in Lemma~\ref{RemarkS6}, it is easy to see that $\phi_{\g}$ is a Lie algebra isomorphism of $(\g,[\cdot,\cdot]_{\rm Lie})$.

We have a unique simply connected Lie group $(G,\cdot)$ such that $(\g,[\cdot,\cdot]_{\rm Lie})$ is the Lie algebra of~$(G,\cdot)$. Since $\phi_{\g}$ is a Lie algebra isomorphism of $(\g,[\cdot,\cdot]_{\rm Lie})$ and $G$ is a simply connected Lie group, we have a unique isomorphism $\Phi$ of the Lie group $(G,\cdot)$ such that $\Phi_{*e}=\phi_{\g}$. By Example~\ref{rmk:00}, the tuple $(G,\diamond, e_{\Phi},\Phi)$ is a Hom-Lie group. Finally, by Lemma~\ref{lem:Des1}, it follows that
\begin{equation*}
[x,y]_{\g^{!}}=\Phi_{*e}[x,y]_{\rm Lie}=\phi_{\g}[x,y]_{\rm Lie}=[x,y]_{\g},\qquad \phi_{\g^{!}}(x)=\Phi_{*e}(x)=\phi_{\g}(x),\qquad \forall\, x,y\in\g^{!},
\end{equation*}
which implies that $\phi_{\g^{!}}=\Phi_{*e_\Phi}$.
\end{proof}

\subsection{One-parameter Hom-Lie subgroups}

Let $(G,\diamond,e_\Phi,\Phi)$ be a Hom-Lie group and $\big(\g^{!},[\cdot,\cdot]_{\g^{!}},\phi_{\g^{!}}\big)$ be its Hom-Lie algebra. Then we define one-parameter Hom-Lie subgroups of $(G,\diamond,e_\Phi,\Phi)$ and prove that there is a one-to-one correspondence between elements of $\g^{!}$ and one-parameter Hom-Lie subgroups of $(G,\diamond,e_\Phi,\Phi)$.

\begin{defi}
A weak homomorphism of Hom-Lie groups
\[ \sigma^{!}\colon \ (R, +, 0, \Id)\rightarrow (G,\diamond,e_\Phi,\Phi)\] is called a one-parameter Hom-Lie subgroup of the Hom-Lie group $(G,\diamond,e_\Phi,\Phi)$.
\end{defi}

\begin{thm}\label{thm:imp1}
Let $(G,\diamond,e_\Phi,\Phi)$ be a Hom-Lie group. Then $\sigma^{!}\colon (R, +, 0, \Id)\rightarrow (G,\diamond,e_\Phi,\Phi)$ is a~one-parameter Hom-Lie subgroup of the Hom-Lie group $(G,\diamond,e_\Phi,\Phi)$ if and only if there exists a unique $x\in \g^{!}$ such that $\sigma^{!}(t)=\Phi(\exp(tx)),$ where $\exp$ is the exponential map of the Lie group $(G,\cdot,e_\Phi)$.
\end{thm}

\begin{proof} For all $x\in \g^{!}$, we have
\begin{equation*}
\Phi\big(\sigma^{!}(t+s)\big)=\Phi\big(\Phi\big(\exp\big((t+s)x\big)\big)\big)=\Phi(\exp(tx)\diamond \exp(sx))=\sigma^{!}(t)\diamond \sigma^{!}(s),
\end{equation*}
which implies that $\sigma^{!}(t)=\Phi\big(\exp(tx)\big)$ is a one-parameter Hom-Lie subgroup of the Hom-Lie group $(G,\diamond,e_\Phi,\Phi)$. For different $x\in \g^!$, we get a different one-parameter Hom-Lie subgroup.

Now, let us assume that $\sigma^{!}\colon (R, +, 0, \Id)\rightarrow (G,\diamond,e_\Phi,\Phi)$ is a one-parameter Hom-Lie subgroup of the Hom-Lie group $(G,\diamond,e_\Phi,\Phi)$. Then, for all $t,s\in\mathbb{R}$,
\begin{equation*}
\Phi^{-1}\big(\sigma^{!}(t+s)\big)=\Phi^{-2}\big(\sigma^{!}(t)\diamond \sigma^{!}(s)\big)=\Phi^{-1}\big(\sigma^{!}(t)\cdot\sigma^{!}(s)\big)=\Phi^{-1}\big(\sigma^{!}(t)\big)\cdot\Phi^{-1}\big(\sigma^{!}(s)\big),
\end{equation*}
which implies that $\Phi^{-1}\big(\sigma^{!}(t)\big)$ is a one-parameter Lie subgroup of the Lie group $(G,\cdot,e_\Phi)$ (defined in Remark~\ref{Rem:LIE-HLIE}). Thus there exists a unique $x\in \g^{!}$, such that $\sigma^{!}(t)=\Phi(\exp (tx))$. The proof is finished.
\end{proof}

By Theorem \ref{thm:imp1}, one-parameter Hom-Lie subgroups of Hom-Lie group $(G,\diamond,e_\Phi,\Phi)$ are in one-to-one correspondence with $\g^{!}$. We denote by $\sigma^{!}_{x}(t)$ the one-parameter Hom-Lie subgroup of the Hom-Lie group $(G,\diamond,e_\Phi,\Phi)$, which corresponds with $x$.

\subsection[The Hexp map]{The $\boldsymbol{\Hexp}$ map}
Let $(G,\diamond,e_\Phi,\Phi)$ be a Hom-Lie group and $\mathfrak{g}^{!}$ be the fibre of the pullback bundle $\Phi^{!}TG$ at $e_\Phi$. Also, let us assume that $\sigma^{!}\colon (R, +,0, \Id)\rightarrow (G,\diamond,e_\Phi,\Phi)$ is a one-parameter Hom-Lie subgroup of the Hom-Lie group $(G,\diamond,e_\Phi,\Phi)$.

Then, let us define a map $\Hexp\colon \mathfrak{g}^{!} \rightarrow G$ by
\begin{equation}
\label{eq:35}\Hexp(x)=\sigma^{!}_{x}(1),\qquad\forall\, x\in \mathfrak{g}^{!}.
\end{equation}

\begin{thm}\label{HLIEprod-HLIEbrac}
Let $(G,\diamond,e_\Phi,\Phi)$ be a Hom-Lie group and $\big(\mathfrak{g}^{!},[\cdot,\cdot]_{\g^{!}},\Phi_{\g^{!}}\big)$ be the associated Hom-Lie algebra. Then the Hom-Lie bracket $[\cdot,\cdot]_{\g^{!}}$ can be expressed in terms of the $\Hexp\colon \mathfrak{g}^{!}\rightarrow G$ map as follows:
\[
[x, y]_{\g^{!}}=\frac{{\rm d}}{{\rm d}t}\frac{{\rm d}}{{\rm d}s}\Big|_{t=0,s=0}\big(\Phi^{-3}(\Hexp(sx)\diamond \Hexp(ty)) \diamond \Phi^{-2}(\Hexp(-sx))\big) \diamond\Phi^{-1}( \Hexp(-ty)),\]
for any $x,y\in \mathfrak{g}^{!}$.
\end{thm}

\begin{proof} Let us denote
\begin{equation*}
\label{eq:36}\Omega_{(x, y)}(t, s):=\big(\Phi^{-3}(\Hexp(sx)\diamond \Hexp(ty)) \diamond \Phi^{-2}(\Hexp(-sx))\big) \diamond \Phi^{-1}(\Hexp(-ty))
\end{equation*}
for all $x,y\in \g^{!}$. From Remark~\ref{Rem:LIE-HLIE}, the triple $(G,\cdot,e_{\Phi})$ is a Lie group and $\g=\g^{!}$, where $(\g,[\cdot,\cdot]_{\g})$ is the Lie algebra of the Lie group $(G,\cdot,e_{\Phi})$. Next, we use~\eqref{eq:35}, Theorem~\ref{thm:imp1}, and Lemma~\ref{lem:Des1} to obtain the following expression:
\begin{gather*}
\frac{{\rm d}}{{\rm d}t}\frac{{\rm d}}{{\rm d}s}\Big|_{t=0,s=0}\Omega_{(x, y)}(t, s)
 = \frac{{\rm d}}{{\rm d}t}\frac{{\rm d}}{{\rm d}s}\Big|_{t=0,s=0}\Phi(\exp(sx)\cdot\exp(ty)\cdot\exp(-sx)\cdot\exp(-ty))\\
\hphantom{\frac{{\rm d}}{{\rm d}t}\frac{{\rm d}}{{\rm d}s}\Big|_{t=0,s=0}\Omega_{(x, y)}(t, s)}{} = \Phi_{*e_\Phi }[x,y]_{\g} =[x, y]_{\g^{!}}.
\end{gather*}
The proof is finished.
\end{proof}

\begin{ex}Let $V$ be a real vector space. Let us recall from Example~\ref{HGrp:GL(V)} that the tuple $({\rm GL}(V),\diamond,\beta, \Ad_{\beta})$ is a Hom-Lie group. Then the triple $\big(\gl(V), [\cdot,\cdot]_{\gl(V)}, \Psi_{\gl(V)}\big)$ is the associated Hom-Lie algebra, where the bracket is given by
\[ [x,y]_{\gl(V)}=\beta\circ x\circ\beta^{-1} \circ y\circ\beta^{-1}-\beta\circ y\circ\beta^{-1}\circ x\circ \beta^{-1},\]
and $ \Psi_{\gl(V)}(x)=\beta\circ x \circ\beta^{-1}$.
\end{ex}

\begin{pro}\label{thm:+}
A map $f\colon (G,\diamond_G,e_\Phi,\Phi)\rightarrow (H,\diamond_H,e_\Psi,\Psi)$ is a weak homomorphism of Hom-Lie groups if and only if $f\colon (G,\cdot_G,e_\Phi)\rightarrow (H,\cdot_H,e_\Psi)$ is a Lie group homomorphism.
\end{pro}
\begin{proof} Let us assume that $f\colon (G,\diamond_G,e_\Phi,\Phi)\rightarrow (H,\diamond_H,e_\Psi,\Psi)$ is a weak homomorphism of Hom-Lie groups. Then, we have
\begin{equation}
\label{eq:1000}f\big(\Phi^{-1}(a)\diamond_G\Phi^{-1}(b)\big)=\Psi^{-1}\big(f(a)\diamond f(b)\big),\qquad\forall\, a,b\in G,
\end{equation}
which implies that $f\colon (G,\cdot_G,e_\Phi)\rightarrow (H,\cdot_H,e_\Psi)$ is a Lie group homomorphism.

Conversely, let $f\colon (G,\cdot_G,e_\Phi)\rightarrow (H,\cdot_H,e_\Psi)$ be a Lie group homomorphism. Then, by~\eqref{eq:1000}, we deduce that~$f$ is a weak homomorphism from $(G,\diamond_G,e_\Phi,\Phi)$ to $(H,\diamond_H,e_\Psi,\Psi)$.
\end{proof}

\begin{thm}\label{thm:99}
Let $f\colon (G,\diamond_G,e_\Phi,\Phi)\rightarrow (H,\diamond_H,e_\Psi,\Psi)$ be a weak homomorphism of Hom-Lie groups. For any $x\in \g^{!}$, we have
\begin{equation*}
f\circ\sigma^{!}_{x}=\sigma^{!}_y,
\end{equation*}
where $y=\Psi^{-1}_{*e_\Psi}f_{*e_\Phi}\Phi_{*e_\Phi}(x)$ and $\sigma^{!}_x$, $\sigma^{!}_y$ are the one-parameter Hom-Lie subgroups of the Hom-Lie groups $(G,\diamond_G,e_\Phi,\Phi)$ and $(H,\diamond_H,e_\Psi,\Psi)$ determined by~$x$ and~$y$ respectively.
\end{thm}
\begin{proof} By the definition of the one-parameter Hom-Lie subgroup $\sigma^{!}_{x}$, it follows that
\begin{equation*}
f\big(\sigma^{!}_{x}(t+s)\big)=f\big(\Phi^{-1}\big(\sigma^{!}_{x}(t)\big)\diamond_G\Phi^{-1}\big(\sigma^{!}_{x}(s)\big)\big) =\Psi^{-1}\big(f\big(\big(\sigma^{!}_{x}(t)\big)\big)\diamond_Hf\big(\big(\sigma^{!}_{x}(s)\big)\big)\big),
\end{equation*}
i.e.,
\begin{equation*}
\Psi\circ f\circ\sigma^{!}_{x}(t+s)= f\circ\sigma^{!}_{x}(t)\diamond_Hf\circ\sigma^{!}_{x}(s).
\end{equation*}
Thus, $f\circ\sigma^{!}_{x}$ is a one-parameter Hom-Lie subgroup of the Hom-Lie group $(H,\diamond_H,e_\Psi,\Psi).$

From Theorem \ref{thm:imp1}, we obtain the following expressions
\begin{equation*}
f\circ\sigma^{!}_{x}(t)=f\big(\Phi(\exp(tx))\big),
\end{equation*}
and
\begin{equation*}
f\big(\Phi(\exp(tx))\big)=\exp\big(tf_{*e_\Phi}\Phi_{*e_\Phi}(x)\big)=\Psi\big(\exp\big(t\Psi^{-1}_{*e_\Psi}f_{*e_\Phi}\Phi_{*e_\Phi}(x)\big)\big), \end{equation*}
which implies that $f\circ\sigma^{!}_{x}(t)=\sigma^{!}_{y}(t),$ where $y=\Psi^{-1}_{*e_\Psi}f_{*e_\Phi}\Phi_{*e_\Phi}(x)$.
\end{proof}

Let $f\colon (G,\diamond_G,e_\Phi,\Phi)\rightarrow (H,\diamond_H,e_\Psi,\Psi)$ be a weak homomorphism of Hom-Lie groups. Then, we define a map $f_{\triangleright}\colon \g^{!}\rightarrow\h^{!}$ by
\begin{equation*}
f_{\triangleright}(x)=\Psi^{-1}_{*e_\Psi}f_{*e_\Phi}\Phi_{*e_\Phi}(x),\qquad\forall\, x\in \g^{!}.
\end{equation*}
\begin{thm}\label{thm:corres-weak morphisms}
With the above notations, the map $f_{\triangleright}\colon \big(\g^{!},[\cdot,\cdot]_{\g^{!}},\phi_{\g^{!}}\big) \rightarrow\big(\h^{!},[\cdot,\cdot]_{\h^{!}},\psi_{\h^{!}}\big)$ is a~weak homomorphism of Hom-Lie algebras.
\end{thm}
\begin{proof} For all $x, y\in \g^{!}$, it follows that
\begin{equation*}
f_{\triangleright}\big([x, y]_{\g^{!}}\big)=\Psi^{-1}_{*e_\Psi}f_{*e_\Phi}\Phi_{*e_\Phi}\big([x, y]_{\g^{!}}\big)=\Psi^{-1}_{*e_\Psi}f_{*e_\Phi}\Phi^{2}_{*e_\Phi}([x, y]).
\end{equation*}
By using Proposition \ref{thm:+}, we have
\begin{gather*}
\Psi^{-1}_{*e_\Psi}f_{*e_\Phi}\Phi^{2}_{*e_\Phi}([x, y]) = \big[\Psi^{-1}_{*e_\Psi}f_{*e_\Phi}\Phi^{2}_{*e_\Phi}(x), \Psi^{-1}_{*e_\Psi}f_{*e_\Phi}\Phi^{2}_{*e_\Phi}(y)\big]\\
\hphantom{\Psi^{-1}_{*e_\Psi}f_{*e_\Phi}\Phi^{2}_{*e_\Phi}([x, y])}{} = \Psi_{*e_\Psi}\big[\Psi^{-1}_{*e_\Psi}f_{\triangleright}\Phi_{*e_\Phi}(x), \Psi^{-1}_{*e_\Psi}f_{\triangleright}\Phi_{*e_\Phi}(y)\big],
\end{gather*}
i.e.,
\begin{equation*}\label{eq:7}
f_{\triangleright}([x, y]_{\Phi})=\big[\Psi^{-1}_{*e_\Psi}f_{\triangleright}\Phi_{*e_\Phi}(x), \Psi^{-1}_{*e_\Phi}f_{\triangleright}\Phi_{*e_\Phi}(y)\big]_{\Psi}.
\end{equation*}
Thus,
$
\psi f_{\triangleright}([x, y]_{\Phi})=[f_{\triangleright}\phi(x),f_{\triangleright}\phi(y)]_{\Psi},
$
which implies that $f_{\triangleright}\colon \big(\g^{!},[\cdot,\cdot]_{\Phi},\phi\big)\rightarrow \big(\h^{!},[\cdot,\cdot]_\Psi,\psi\big)$ is a~weak homomorphism of Hom-Lie algebras.
\end{proof}

\begin{thm}[universality of the Hexp map]\label{thm:after}
Let $f\colon (G,\diamond_G,e_\Phi,\Phi)\rightarrow(H,\diamond_H,e_\Psi,\Psi)$ be a~weak homomorphism of Hom-Lie groups. Then,
\begin{equation*}\label{eq:678}
f\big(\Hexp(x)\big)=\Hexp\big(f_{\triangleright}(x)\big),\qquad\forall\, x\in\g^{!},
\end{equation*}
i.e., the following diagram commutes:
\[
 \xymatrix{
 G \ar[r]^{f}
 & H \\
 \g^{!} \ar[u]^{\Hexp} \ar[r]^{f_{\triangleright}}
 & \h^{!}. \ar[u]_{\Hexp} }
\]
\end{thm}
\begin{proof} By the definition of $\Hexp$, it follows that
\begin{equation*}
\Hexp\big(f_{\triangleright}(x)\big)=\Psi\big(\exp\big({\Psi^{-1}_{*e_\Psi}f_{*e_\Phi}\Phi_{*e_\Phi}x}\big)\big),
\end{equation*}
and
\begin{equation*}
f\big(\Hexp(x)\big)=f(\Phi(\exp(x)))=\Psi\big(\exp\big({\Psi^{-1}_{*e_\Psi}f_{*e_\Phi}\Phi_{*e_\Phi}x}\big)\big).
\end{equation*}
Thus we have $f\big(\Hexp(x)\big)=\Hexp\big(f_{\triangleright}(x)\big),$ for all $~x\in\g^{!}$.
\end{proof}

\section{Actions of Hom-Lie groups and Hom-Lie algebras}\label{actionhom}

Let $(G,\diamond_G,e_\Phi,\Phi)$ be a Hom-Lie group and $M$ be a~smooth manifold. Let $\theta\colon G\times M\rightarrow M$ be a~smooth map that we denote by
 \begin{equation*}
 \theta(a, x)=a\odot x,\qquad \forall\, a\in G, x\in M.
 \end{equation*}
\begin{defi}

 The map $\theta\colon G\times M\rightarrow M$ is called an {\it action} of the Hom-Lie group $(G,\diamond_G,e_\Phi,\Phi)$ on the smooth manifold $M$ with respect to a map $\iota\in\Diff(M)$ if the following conditions are satisfied:
 \begin{itemize}\itemsep=0pt
 \item[(i)] $e\odot x=\iota(x)$, $\forall\, x\in M$;
 \item[(ii)] $(a\diamond_Gb)\odot x=\Phi(a)\odot\big(\iota^{-1}(\Phi(b)\odot x)\big)$, $\forall\, a,b\in G$, $x\in M$.
 \end{itemize}
We denote this action by $(G, \theta, M, \iota)$.
\end{defi}
For all $a\in G$, define $L_{a}\colon M\rightarrow M$ by
\begin{equation*}
L_{a}(x)=a\odot x=\theta(a,x),\qquad \forall\, x\in M.
\end{equation*}
Since $L_{e_\Phi}=\iota\in\Diff(M)$ and $L_{\Phi^{-1}(a\diamond a^{-1})}=L_{a}\circ\iota^{-1}\circ L_{a^{-1}}=\iota$, we have $L_{a}\circ\iota^{-1}\in\Diff(M)$, and thus $L_a\in\Diff(M).$ Let us define a map $L\colon G\rightarrow \Diff(M)$ by $L(a)=L_{a}$ for all $a\in G$.
\begin{thm}
With the above notations, the map $\theta\colon G\times M\rightarrow M$ is an action of the Hom-Lie group $(G,\diamond_G,e_\Phi,\Phi)$ on $M$ with respect to $\iota\in \Diff(M)$ if and only if the map $L\colon G\rightarrow \Diff(M)$ is a weak homomorphism from the Hom-Lie group $(G,\diamond_G,e_\Phi,\Phi)$ to the Hom-Lie group $(\Diff(M),\diamond,\iota,\Ad_{\iota}).$
\end{thm}

\begin{proof} Let us first assume that the map $\theta\colon G\times M\rightarrow \Diff(M)$ is an action of the Hom-Lie group $(G,\diamond_G,e_\Phi,\Phi)$ on $M$ with respect to the map $\iota\in \Diff(M)$. Then we have
\begin{equation*}
L(e_\Phi)=L_{e_\Phi}=\iota, \qquad L(a\diamond b)(x)=\theta(a\diamond b,x)=L(\Phi(a))\circ\iota^{-1}\circ L(\Phi(b))(x).
\end{equation*}
Thus, we have
\begin{equation*}
\Ad_{\iota}\circ L(a\diamond b)=L(\Phi(a))\diamond L(\Phi(b)),
\end{equation*}
which implies that the map $L\colon (G,\diamond_G,e_\Phi,\Phi)\rightarrow(\Diff(M),\diamond,\iota,\Ad_{\iota})$ is a weak homomorphism of Hom-Lie groups.

Conversely, let us assume that $L\colon (G,\diamond_G,e_\Phi,\Phi)\rightarrow(\Diff(M),\diamond,\iota,\Ad_{\iota})$ is a weak homomorphism of Hom-Lie groups. Then, it follows that
\begin{equation}\label{eq:use2}
\theta(e_\Phi,x)=L(e_\Phi)(x)=\iota(x),\qquad \forall\, x\in M,
\end{equation}
and
\begin{equation*}
\big(\Ad_{\iota}\circ L(a\diamond b)\big)(x)=\iota\circ L(\Phi(a))\circ\iota^{-1}\circ L(\Phi(b))\circ\iota^{-1}(x)=\big(\iota\circ L(a\diamond b)\circ\iota^{-1}\big)(x),
\end{equation*}
which implies that
\[ L(a\diamond b)(x)=L(\Phi(a))\circ\iota^{-1}\circ L(\Phi(b))(x).\]
Therefore, we get the following identity:
\begin{equation}\label{eq:use3}
\theta(a\diamond b,x)=\Phi(a)\odot\big(\iota^{-1}(\Phi(b)\odot x)\big).
\end{equation}
By \eqref{eq:use2} and \eqref{eq:use3}, we deduce that $\theta\colon G\times M \rightarrow M$ is an action of the Hom-Lie group $(G,\diamond_G,e_\Phi,\Phi)$ on the smooth manifold $M$ with respect to $\iota\in\Diff(M)$.
\end{proof}

If $(G,\diamond_G,e_\Phi,\Phi)$ is a Hom-Lie group, then let us define a map $\widetilde{\Ad}\colon G\times G\rightarrow G$ by
\begin{equation*}\label{eq:defi1}
\widetilde{\Ad}(a,b)=\Phi^{-1}(a\diamond_G b)\diamond_G a^{-1},\qquad \forall\, a,b\in G.
\end{equation*}

\begin{lem}\label{lem:ad-action}
The map $\widetilde{\Ad}\colon G\times G\rightarrow G$ gives an action of the Hom-Lie group $(G,\diamond_G,e_\Phi,\Phi)$ on~$G$ with respect to the map $\Phi\in \Diff(G)$.
\end{lem}
\begin{proof} For all $x\in G$, we have
\begin{equation*}
\widetilde{\Ad}(e_{\Phi},x)=\Phi^{-1}(e_{\Phi}\diamond_G x)\diamond_G e^{-1}_{\Phi}=\Phi(x),
\end{equation*}
and
\begin{equation*}
\widetilde{\Ad}(a\diamond_G b,x)=\Phi^{-1}\big((a\diamond_G b)\diamond_G x\big)\diamond_G (a\diamond_G b)^{-1},\qquad \forall\, a,b\in G.
\end{equation*}
Let us denote $a\odot b:=\widetilde{\Ad}(a,b)$, then we get the following expression:
\begin{align*}
\Phi(a)\odot\big(\Phi^{-1}(\Phi(b)\odot x)\big)
&=\Phi(a)\odot\big(\Phi^{-1}\big(\Phi^{-1}(\Phi(b)\diamond_G x)\diamond_G (\Phi(b))^{-1}\big)\big)\\
&=\Phi(a)\odot\big(\big(\Phi^{-1}(b)\diamond_G \Phi^{-2}(x)\big)\diamond_G b^{-1}\big)\\
&=\Phi^{-1}\big(\Phi(a)\diamond_G \big(\big(\Phi^{-1}(b)\diamond_G \Phi^{-2}(x)\big)\diamond_G b^{-1}\big)\diamond_G(\Phi(a))^{-1}\\
&= \big(a\diamond_G \big(\Phi^{-1}(b)\diamond_G \big(\Phi^{-3}(x)\diamond_G \big(\Phi^{-2}(b)\big)^{-1}\big)\big)\big)\diamond_G(\Phi(a))^{-1}\\
&=\big(\Phi^{-1}(a\diamond_G b)\diamond_G \big(\Phi^{-2}(x)\diamond_G \big(\Phi^{-1}(b)\big)^{-1}\big)\big)\diamond_G \Phi(a)^{-1}\\
&=\big(\big(\Phi^{-2}(a\diamond_G b)\diamond_G \Phi^{-2}(x)\big)\diamond_G b^{-1}\big)\diamond_G \Phi(a)^{-1}\\
&=\Phi^{-1}\big((a\diamond_G b)\diamond x\big)\diamond_G (a\diamond b)^{-1},
\end{align*}
which implies that
\begin{gather*} \widetilde{\Ad}(a\diamond_G b,x)=\Phi(a)\odot\big(\Phi^{-1}(\Phi(b)\odot x)\big)\\
\hphantom{\widetilde{\Ad}(a\diamond_G b,x)}{}
=\widetilde{\Ad}\big(\Phi(a),\big(\Phi^{-1}\big(\widetilde{\Ad}(\Phi(b),x)\big)\big)\big),\qquad \forall\, a,b,x\in G.
\end{gather*}
Thus, the map $\widetilde{\Ad}\colon G\times G\rightarrow G$ gives an action of the Hom-Lie group $(G,\diamond_G,e_\Phi,\Phi)$ on the underlying manifold~$G$ with respect to the map $\Phi\in \Diff(G)$.
\end{proof}

\begin{defi}
Let $(G,\diamond_G,e_\Phi,\Phi)$ be a Hom-Lie group, $V$ be a vector space, and $\beta\in {\rm GL}(V)$. Then, a weak homomorphism of Hom-Lie groups \[ \rho\colon \ (G,\diamond_G,e_\Phi,\Phi)\rightarrow ({\rm GL}(V),\diamond,\beta,\Ad_{\beta})\] is called a {\it representation} of the Hom-Lie group $(G,\diamond_G,e_\Phi,\Phi)$ on the vector space $V$ with respect to $\beta\in {\rm GL}(V)$.
\end{defi}

Let $\big(\g^{!},[\cdot,\cdot]_{!},\phi_{\g^{!}}\big)$ be the Hom-Lie algebra of a Hom-Lie group $(G,\diamond_G,e_\Phi,\Phi)$. From Lem\-ma~\ref{lem:ad-action}, $\widetilde{\Ad}$ gives an action of the Hom-Lie group $(G,\diamond_G,e_\Phi,\Phi)$ on $G$ with respect to the map $\Phi$. Now, let us denote $\widetilde{\Ad}_{a}:=\widetilde{\Ad}(a,\cdot)$ for any $a\in G$. Then we observe that for all $a\in G$, the map $\widetilde{\Ad}_{a}\colon G\rightarrow G$ is a weak isomorphism of Hom-Lie groups. Let us denote by $\big(\widetilde{\Ad}_{a}\big)_\triangleright\colon \g^{!}\rightarrow \g^{!}$, the weak isomorphism of Hom-Lie algebra $\big(\g^{!},[\cdot,\cdot]_{!},\phi_{\g^{!}}\big)$, obtained by Theorem~\ref{thm:corres-weak morphisms}. Subsequently, we have the following lemma.

 \begin{lem}\label{lem:Ad-rep}
The map $\widehat{\Ad}\colon G\rightarrow {\rm GL}\big(\g^{!}\big)$, defined by
\begin{equation*}
\widehat{\Ad}(a)=\big(\widetilde{\Ad}_{a}\big)_\triangleright,\qquad \forall\, a\in G
\end{equation*}
is a weak homomorphism from $(G,\diamond_G,e_\Phi,\Phi)$ to $\big({\rm GL}\big(\g^{!}\big),\diamond,\phi_{\g^{!}},\Ad_{\phi_{\g^{!}}}\big)$.
\end{lem}
\begin{proof} For all $a, b\in G$ and $x\in\g^{!}$, it follows that
\begin{align*}
\widehat{\Ad}(a\diamond_G b)(x)&= \big(\widetilde{\Ad}_{a\diamond_G b}\big)_\triangleright(x)\\
&= \frac{{\rm d}}{{\rm d}t}\Big|_{t=0}\Phi^{-1}\big(\Phi^{-1}((a\diamond_G b)\diamond_G\Phi(\exp(tx)))\diamond_G (a\diamond_G b)^{-1}\big)
\end{align*}
and
\begin{gather*}
 \widehat{\Ad}(\Phi(a))\circ\Phi^{-1}_{*}\circ\widehat{\Ad}(\Phi(b))(x)\\
\qquad {} = \widehat{\Ad}(\Phi(a))\Phi^{-1}_{*}\frac{{\rm d}}{{\rm d}t}\Big|_{t=0}\Phi^{-1}\big(\big(\Phi^{-1}(\Phi(b)\diamond_G\Phi(\exp(tx)))\diamond_G (\Phi(b))^{-1}\big)\big)\\
\qquad {} = \frac{{\rm d}}{{\rm d}t}\Big|_{t=0}\Phi^{-1}\big(\big(\Phi^{-1}\big(\Phi(a)\diamond_G\big(b\diamond_G\big(\Phi^{-1}\big(\exp(tx)\diamond_G b^{-1}\big)\big)\big)\diamond_G (\Phi(a))^{-1}\big)\big)\big)\\
\qquad {} = \frac{{\rm d}}{{\rm d}t}\Big|_{t=0}\big(\Phi^{-1}(a)\diamond_G\big(\Phi^{-2}(b)\diamond_G\big(\Phi^{-3}\big(\exp(tx)\diamond_G b^{-1}\big)\big)\big)\diamond_G a^{-1}\big)\\
\qquad {} = \frac{{\rm d}}{{\rm d}t}\Big|_{t=0} \big(\Phi^{-2}(a\diamond_G b)\diamond_G\big(\Phi^{-2}\big(\exp(tx)\diamond_G b^{-1}\big)\big)\big)\diamond_G a^{-1} \\
\qquad {} = \frac{{\rm d}}{{\rm d}t}\Big|_{t=0}\big(\Phi^{-2}(a\diamond_G b)\diamond_G \Phi^{-1}(\exp(tx))\big)\diamond_G \Phi^{-1}(a\diamond_G b)^{-1},
\end{gather*}
i.e.,
\begin{equation*}
\widehat{\Ad}(a\diamond_G b)=\widehat{\Ad}(\Phi(a))\circ\Phi^{-1}_{*}\circ\widehat{\Ad}(\Phi(b)).
\end{equation*}
Thus, we have
\begin{equation*}
\Phi_{*}\circ\widehat{\Ad}(a\diamond_G b)\circ\Phi^{-1}_{*}=\Phi_{*}\circ\widehat{\Ad}(\Phi(a))\circ\Phi^{-1}_{*}\circ\widehat{\Ad}(\Phi(b))\circ\Phi^{-1}_{*},
\end{equation*}
which implies that $\widehat{\Ad}$ is a Hom-Lie group (weak) homomorphism from $G$ to ${\rm GL}\big(\g^{!}\big)$.
\end{proof}

From Lemma \ref{lem:Ad-rep}, every Hom-Lie group $(G,\diamond_G,e_\Phi,\Phi)$ has a representation on its Hom-Lie algebra $\big(\g^{!},[\cdot,\cdot]_{\g^{!}},\phi_{\g^{!}}\big)$ with respect to the map $\phi_{\g^{!}}$, which we call the ``adjoint representation''.

Let $(G,\theta,M,\iota)$ be an action of the Hom-Lie group $(G,\diamond_G,e_\Phi,\Phi)$. By Theorem~\ref{thm:imp1}, for any $x\in\g^{!}$, we have a unique map $\sigma^{!}_{x}\colon \mathbb{R}\rightarrow G$. Now, we consider the curve $\gamma\colon \mathbb{R}\rightarrow M$ given by
\[ \gamma(t)=\Phi(\exp(tx))\odot m,\qquad \forall\, t\in \mathbb{R}, x\in \g^{!}, m\in M.\]
Then we get a section of the pullback bundle $\iota^{*}TM$ as follows:
\begin{equation*}
\widetilde{x}(m)=\frac{{\rm d}}{{\rm d}t}\bigg|_{t=0}\gamma(t),\qquad \forall\, m\in M.
\end{equation*}
 Thus, we have a map $\kappa\colon \g^{!}\rightarrow \Gamma(\iota^{*}TM)$ given by $\kappa(x)=\widetilde{x}$ for all $x\in \g^{!}$. We denote $\kappa(x)(m)$ simply by $x\bullet m$ for any $x\in \g^{!}$, $m\in M$.

\begin{defi}
The map $\kappa\colon \g^{!}\rightarrow \Gamma(\iota^{*}TM)$ is called the infinitesimal action of the Hom-Lie algebra $\g^{!}$ on the smooth manifold $M$ with respect to the map $\iota\in \Diff(M)$.
\end{defi}
By Lemma \ref{lem:Ad-rep} and Theorem~\ref{thm:corres-weak morphisms}, the map
\[ (\widehat{\Ad})_{\triangleright}\colon \ \big(\g^{!},[\cdot,\cdot]_{\g^{!}},\phi_{\g^{!}}\big)\rightarrow \big(\gl\big(\g^{!}\big),[\cdot,\cdot]_{\phi_{\g^{!}}}, \Ad_{\phi_{\g^{!}}}\big)\]
is a weak homomorphism of Hom-Lie algebras. To simplify the notations, we denote $\widehat{\ad}:=\big(\widehat{\Ad}\big)_{\triangleright}$, then we have the following result.
\begin{thm}
With the above notations, we have
\[\big(\widehat{\ad}(x)\big)\bullet y=[x, y]_{\g^{!}},\qquad \forall\, x,y\in\g^{!}.\]
\end{thm}

\begin{proof} By the definition of the infinitesimal action of Hom-Lie algebra $\g^{!}$, we have
\begin{align*}
\big(\widehat{\ad}(x)\big)\bullet y&= \frac{{\rm d}}{{\rm d}t}\Big|_{t=0}\big(\Phi(\exp(tx))\odot y\big)\\
&= \frac{d}{dt}\Big|_{t=0}\frac{{\rm d}}{{\rm d}s}\Big|_{s=0}\Phi^{-1}\big(\Phi^{-1}\big(\Phi(\exp(tx))\diamond_G\Phi(\exp(sy))\big)\diamond_G(\Phi(\exp(tx)))^{-1}\big)\\
&= \frac{{\rm d}}{{\rm d}t}\Big|_{t=0}\frac{{\rm d}}{{\rm d}s}\Big|_{s=0}\Phi(\exp(tx))\cdot\Phi(\exp(sy))\cdot(\Phi(\exp(tx)))^{-1}
 = \Phi_{*e_\Phi}[x,y],
\end{align*}
which implies that $(\widehat{\ad}(x))\bullet y=[x, y]_{\g^{!}}$.
\end{proof}

\begin{pro}
Let $(G,\diamond_G,e_\Phi,\Phi)$ be a Hom-Lie group and $\g^{!}$ be the Hom-Lie algebra of $G$. Then, we have
\begin{equation*}
\widehat{\Ad}(\Hexp(x))=\Hexp\big(\widehat{\ad}(x)\big),\qquad \forall\, x\in\g^{!}.
\end{equation*}
\end{pro}
\begin{proof} The proof of the proposition follows from Theorem~\ref{thm:after} and Lemma~\ref{lem:Ad-rep}.
\end{proof}

\section[\protect{Integration of the Hom-Lie algebra (gl(V),[.,.]beta,Adbeta)}]{Integration of the Hom-Lie algebra $\boldsymbol{(\gl(V),[\cdot,\cdot]_\beta,\Ad_{\beta})}$}\label{Integrationgl}

Let $V$ be a vector space and $\beta\in {\rm GL}(V)$. Let us define a map $\frke_\beta\colon \gl(V)\rightarrow {\rm GL}(V)$ by
\begin{equation*}
 \frke_\beta^A=\beta+\beta\circ A\circ \beta^{-1} +\frac{1}{2!}\beta\circ A\circ \beta^{-1} \circ A\circ \beta^{-1} +\cdots,\qquad \forall\, A\in\gl(V).
\end{equation*}
Let us denote by $e\colon \gl(V)\rightarrow {\rm GL}(V)$, the usual exponential map. Then the map $\frke_\beta\colon \gl(V)\rightarrow {\rm GL}(V)$ can be written as follows:
 \begin{equation}\label{Hexp-exp}
 \frke_\beta^A=\beta\circ e^{A\circ \beta^{-1}} ,\qquad \forall\, A\in\gl(V).
 \end{equation}

The map $\eb\colon \gl(V)\rightarrow {\rm GL}(V)$ gives rise to a one-parameter Hom-Lie subgroup of the Hom-Lie group $({\rm GL}(V),\diamond,\Ad_{\beta})$. More precisely, for any $A\in\gl(V)$, let us define a map
$\sigma_A\colon \mathbb R\lon {\rm GL}(V)$ by
\begin{equation*}
 \sigma_A(t)=\eb^{tA},\qquad \forall\, t\in\mathbb R.
\end{equation*}
Then we have the following lemma.
\begin{lem}
 With the above notations, $\sigma_A\colon(\mathbb R,+,0,\Id) \rightarrow ({\rm GL}(V),\diamond,\beta,\Ad_{\beta})$ is a weak homomorphism of Hom-Lie groups, i.e.,
\begin{equation*}
 \sigma_A(t+s)=\big(\Ad_{\beta^{-1}}\sigma_A(t)\big)\diamond \big(\Ad_{\beta^{-1}}\sigma_A(s)\big),\qquad \forall\, t,s\in\mathbb R.
 \end{equation*}
\end{lem}
\begin{proof} By \eqref{Hexp-exp}, we have
\begin{align*}
 \sigma_A(t+s)&= \eb ^{(t+s)A}=\beta\circ e^{(t+s)(A\beta^{-1})} =\beta\circ\big(e^{tA\circ \beta^{-1}}e^{sA\circ\beta^{-1}}\big)
 = \eb ^{tA}\circ\beta^{-1}\circ\eb ^{sA}\\
 &= \Ad_{\beta^{-1}}\eb ^{sA}\diamond \Ad_{\beta^{-1}}\eb ^{tA}
 = \big(\Ad_{\beta^{-1}}\sigma_A(t)\big)\diamond \big(\Ad_{\beta^{-1}}\sigma_A(s)\big),
\end{align*}
for all $t,s\in \mathbb{R}$.\end{proof}

Thus, for any $A\in \gl(V)$, the map $\sigma_A\colon (\mathbb R,+,0,\Id) \rightarrow ({\rm GL}(V),\diamond,\beta,\Ad_{\beta})$ is a one-parameter Hom-Lie subgroup of the Hom-Lie group $({\rm GL}(V),\diamond,\beta,\Ad_{\beta})$.

\begin{pro}
Let us consider the Hom-Lie algebra $(\gl(V),[\cdot,\cdot]_{\beta},\Ad_{\beta})$. Then the Hom-Lie bracket $[\cdot,\cdot]_{\beta}$ can be expressed as follows:
\begin{gather}\label{eq:exp1}
[A,B]_\beta=\frac{{\rm d}}{{\rm d}t}\frac{{\rm d}}{{\rm d}s}\Big|_{t=0,s=0}\big(\big(\Ad_{\beta^{-3}}\big(\eb^{sA}\diamond \eb^{tB}\big)\big) \diamond \Ad_{\beta^{-2}}\big(\eb^{-sA}\big)\big) \diamond \Ad_{\beta^{-1}}\big(\eb^{-tB}\big).
\end{gather}
\end{pro}

\begin{proof} By \eqref{Hexp-exp}, it follows that
\begin{gather*}
 \frac{{\rm d}}{{\rm d}t}\frac{{\rm d}}{{\rm d}s}\Big|_{t=0,s=0}\big(\big(\Ad_{\beta^{-3}}\big(\eb^{sA}\diamond \eb^{tB}\big)\big) \diamond \Ad_{\beta^{-2}}\big(\eb^{-sA}\big)\big) \diamond \Ad_{\beta^{-1}}\big(\eb^{-tB}\big)\\
\qquad {}= \frac{{\rm d}}{{\rm d}t}\frac{{\rm d}}{{\rm d}s}\Big|_{t=0,s=0}\big(\beta^{-2}\big(\beta^{3}e^{sA\beta^{-1}}e^{tB\beta^{-1}}e^{-sA\beta^{-1}} e^{-tB\beta^{-1}}\beta^{-2}\big)\beta^{2}\big)\\
\qquad {}=\frac{{\rm d}}{{\rm d}t}\frac{{\rm d}}{{\rm d}s}\Big|_{t=0,s=0}\big(\beta e^{sA\beta^{-1}}e^{tB\beta^{-1}}e^{-sA\beta^{-1}}e^{-tB\beta^{-1}}\big)\\
\qquad {}=\beta A\beta^{-1}B\beta^{-1}-\beta B\beta^{-1}A\beta^{-1}
=[A,B]_\beta,
\end{gather*}
which implies that (\ref{eq:exp1}) hold.
\end{proof}

\section{Integration of the derivation Hom-Lie algebra}

Let the triple $(\g,[\cdot,\cdot]_{\g},\phi_{\g})$ be a Hom-Lie algebra. In this section, we consider the derivation of $(\g,[\cdot,\cdot]_{\g},\phi_{\g})$ and show that it is the Hom-Lie algebra of the Hom-Lie group of automorphisms of $(\g,[\cdot,\cdot]_{\g},\phi_{\g})$.

Let us recall from Definition~\ref{def:Derivation}, the space of derivations of a Hom-Lie algebra $(\g,[\cdot,\cdot]_{\g},\phi_{\g})$ that is denoted by $\Der(\g)$. Then, we have the following proposition.
 \begin{pro} \label{Der}
 Let $(\g, [\cdot, \cdot]_{\g}, \phi_{\g})$ be a Hom-Lie algebra. Then $\big(\Der(\g),[\cdot,\cdot]_{\phi_\g}, \Ad_{\phi_{\g}}\big)$ is a Hom-Lie algebra, which is a subalgebra of the Hom-Lie algebra $\big(\mathfrak{gl}(\frkg),[\cdot,\cdot]_{\phi_{\g}},\Ad_{\phi_{\g}}\big)$ given by Proposition~{\rm \ref{pro:Hom-Lie}}.
 \end{pro}
\begin{proof} The proof is a straightforward verification, and we leave the details to readers.
\end{proof}

\begin{pro}\label{eqv3}
A linear map $D\colon \g\rightarrow \g$ is a derivation of the Hom-Lie algebra $(\g, [\cdot, \cdot]_{\g}, \phi_{\g})$ if and only if the map $D\circ \phi_{\g}^{-1}\colon \g\rightarrow \g$ is a derivation of the Lie algebra $(\g, [\cdot, \cdot]_{\rm Lie})$ given in Lemma~{\rm \ref{RemarkS6}}.
\end{pro}

\begin{proof} Let us assume that $D\colon \g\rightarrow \g$ is a derivation of the Hom-Lie algebra $(\g, [\cdot, \cdot]_{\g}, \phi_{\g})$, i.e.,
\begin{gather*}\label{derivation}
D[x,y]_{\g}=\big[\phi_{\g}(x),\big(\Ad_{\phi_{\g}^{-1}}D\big)(y)\big]_{\g}+\big[\big(\Ad_{\phi_{\g}^{-1}}D\big)(x),
\phi(y)\big]_{\g},\qquad \forall\, x,y\in\mathfrak{g}.
\end{gather*}
We need to show that $D\circ \phi_{\g}^{-1}\colon \g\rightarrow \g$ is a derivation of the Lie algebra $(\g, [\cdot, \cdot]_{\rm Lie})$. Equivalently, we need to show the following identity:
\begin{equation}\label{req eq.}
D\circ \phi_{\g}^{-1}[x, y]_{\rm Lie}= \big[D\big(\phi_{\g}^{-1}(x)\big), y\big]_{\rm Lie}+\big[x, D\big(\phi_{\g}^{-1}(y)\big)\big]_{\rm Lie},\qquad \forall\, x,y\in\mathfrak{g}.
\end{equation}
The left hand side of~\eqref{req eq.} can be written as
\begin{align*}
D\circ \phi_{\g}^{-1}[x, y]_{\rm Lie}&= D\big[\phi_{\g}^{-2}(x), \phi_{\g}^{-2}(y)\big]_{\g}\\
& =\big[\phi_{\g}^{-1}(x), \phi_{\g}^{-1}\circ D\circ \phi_{\g}^{-1}(y)\big]_{\g}+\big[\phi_{\g}^{-1}\circ D\circ \phi_{\g}^{-1}(x), \phi_{\g}^{-1}(y)\big]_{\g},
\end{align*}
and the right hand side of~\eqref{req eq.} can be written as
 \begin{gather*}
\big[D(\phi_{\g}^{-1}(x)), y\big]_{\rm Lie}+\big[x, D(\phi_{\g}^{-1}(y))\big]_{\rm Lie} \\
\qquad{} =
\big[\phi_{\g}^{-1}\circ D\circ \phi_{\g}^{-1}(x), \phi_{\g}^{-1}(y)\big]_{\g} +\big[\phi_{\g}^{-1}(x), \phi_{\g}^{-1}\circ D\circ \phi_{\g}^{-1}(y)\big]_{\g}.
\end{gather*}
Thus, it implies that \eqref{req eq.} holds, i.e., $D\circ \phi_{\g}^{-1}$ is a derivation of the Lie algebra $(\g, [\cdot, \cdot]_{\rm Lie})$.

Conversely, if $D\circ \phi_{\g}^{-1}$ is a derivation of the Lie algebra $(\g, [\cdot, \cdot]_{\rm Lie})$, then it easily follows that~$D$ is a derivation of the Hom-Lie algebra $(\g, [\cdot, \cdot]_{\g},\phi_{\g})$.
\end{proof}

\begin{defi}
Let $(\g,[\cdot,\cdot]_{\g},\phi_{\g} )$ be a Hom-Lie algebra. An {\it automorphism} of $(\g,[\cdot,\cdot]_{\g},\phi_{\g} )$ is a~map $\theta\in {\rm GL}(\g)$ such that
\begin{equation*}
\label{eq:42} \theta[x,y]_{\g}=\big[\Ad_{\phi_{\g}^{-1}}\theta (x),\Ad_{\phi_{\g}^{-1}}\theta(y)\big]_{\g},\qquad \forall\, x,y\in\g.
\end{equation*}
\end{defi}

We denote the set of automorphisms of a Hom-Lie algebra $(\g,[\cdot,\cdot]_{\g},\phi_{\g} )$ by $\Aut(\g)$. In fact, there is a Hom-Lie group structure on the set~$\Aut(\g)$.

\begin{pro}\label{proposition6.3}
 Let $(\g,[\cdot,\cdot]_{\g},\phi_{\g} )$ be a Hom-Lie algebra. Then the tuple $\big(\Aut(\g),\diamond,\phi_{\g},\Ad_{\phi_{\g}}\big)$ is a~Hom-Lie subgroup of the Hom-Lie group $\big({\rm GL}(\g),\diamond,\phi_{\g},\Ad_{\phi_{\g}}\big)$.
\end{pro}

\begin{proof} First, we show that the structure map $\phi_{\g}$ is an automorphism of the Hom-Lie algebra $(\g,[\cdot,\cdot]_{\g},\phi_{\g} )$. It follows from the following expression:
\[ \phi_{\g}[x, y]_{\g}=[\phi_{\g}(x), \phi_{\g}(y)]_{\g}=\big[\Ad_{\phi_{\g}^{-1}}\phi_{\g} (x),\Ad_{\phi_{\g}^{-1}}\phi_{\g}(y)\big]_{\g}\]
for all $x,y\in\g$. Now, let $\theta\in\Aut(\g)$, then we have
\begin{equation*}
\theta\big[\Ad_{\phi_{\g}^{-1}}\theta^{-1}(x), \Ad_{\phi_{\g}^{-1}}\theta^{-1}(y)\big]_{\g}=[x, y]_{\g},
\end{equation*}
which implies that $\theta^{-1}[x, y]_{\g}=\big[\Ad_{\phi_{\g}^{-1}}\theta^{-1}(x), \Ad_{\phi_{\g}^{-1}}\theta^{-1}(y)\big]_{\g}$ for all $x,y\in\g$. Thus $\theta^{-1}\in\Aut(\g)$.

Moreover,
\begin{align*}
\theta_{1}\diamond\theta_{2}[x, y]_{\g}&= \phi_{\g}\circ\theta_{1}\circ\phi_{\g}^{-1}\circ\theta_{2}\circ\phi_{\g}^{-1}[x,y]_{\g}
 = \phi_{\g}\circ\theta_{1}\circ\phi_{\g}^{-1}\big[\phi_{\g}^{-1}\circ\theta_{2} (x),\phi_{\g}^{-1}\circ\theta_{2}(y)\big]_{\g}\\
&= \big[\theta_{1}\circ\phi_{\g}^{-1}\circ\theta_{2}(x), \theta_{1}\circ\phi_{\g}^{-1}\circ\theta_{2}\circ\phi_{\g}(y)\big]_{\g}\\
&= \big[\Ad_{\phi_{\g}^{-1}}(\theta_{1}\diamond\theta_{2}) (x),\Ad_{\phi_{\g}^{-1}}(\theta_{1}\diamond\theta_{2})(y)\big]_{\g},
\end{align*}
for all $\theta_{1}, \theta_{2} \in \Aut(\g)$, and $x,y\in\g$. Therefore, $\theta_{1}\diamond\theta_{2} \in \Aut(\g).$ Finally, we have
\begin{align*}
\Ad_{\phi_{\g}}\theta[x,y]_{\g} = \phi_{\g}\circ\theta\circ\phi_{\g}^{-1}[x, y]_{\g}=[\theta(x), \theta(y)]_{\g}=\big[\Ad_{\phi_{\g}^{-1}}(\Ad_{\phi_{\g}}\theta)(x), \Ad_{\phi_{\g}^{-1}}(\Ad_{\phi_{\g}}\theta)(y)\big]_{\g},
\end{align*}
which implies that $\Ad_{\phi_{\g}}\theta \in \Aut({\g})$. Hence, $\big(\Aut(\g),\diamond,\phi_{\g},\Ad_{\phi_{\g}}\big)$ is a Hom-Lie subgroup of the Hom-Lie group $\big({\rm GL}(\g),\diamond,\phi_{\g},\Ad_{\phi_{\g}}\big)$.
\end{proof}

\begin{thm}\label{thm:deraut}
 With the above notations, the triple $\big(\Der(\g),[\cdot,\cdot]_{\phi_{\g}},Ad_{\phi_{\g}}\big)$ is the Hom-Lie algebra of the Hom-Lie group $\big(\Aut(\g),\diamond,\phi_{\g},\Ad_{\phi_{\g}}\big)$.
\end{thm}
\begin{proof} Let us first assume that $D$ is a derivation of the Hom-Lie algebra $(\g,[\cdot,\cdot]_{\g},\phi_{\g} )$. By~\eqref{eq:35}, we have $\Hexp(tD)=\phi_{\g} e^{tD\phi_{\g}^{-1}}$. Then we prove that $\Hexp(tD)$ is an automorphism of the Hom-Lie algebra $(\g, [\cdot, \cdot]_{\g}, \phi_{\g})$. For all $x,y\in \g$, we get the following equation:
\begin{equation}
\label{eq:43}\Hexp(tD)[x, y]_{\g}=\Hexp(tD)[\phi_{\g}(x), \phi_{\g}(y)]_{\rm Lie}=\phi_{\g} e^{tD\phi_{\g}^{-1}}[\phi_{\g}(x), \phi_{\g}(y)]_{\rm Lie}.
\end{equation}
By Proposition \ref{eqv3}, the map $D\circ\phi_{\g}^{-1}\colon \g\rightarrow \g$ is a derivation of Lie algebra $(\g,[\cdot,\cdot]_{\rm Lie})$. This, in turn, implies that $e^{tD\phi_{\g}^{-1}}$ is an automorphism of $(\g,[\cdot,\cdot]_{\rm Lie})$. Thus,
\begin{equation*}
\phi_{\g} e^{tD\phi_{\g}^{-1}}[\phi_{\g}(x), \phi_{\g}(y)]_{\rm Lie}=\phi_{\g}\big[e^{tD\phi_{\g}^{-1}}\phi_{\g}(x), e^{tD\phi_{\g}^{-1}}\phi_{\g}(y)\big]_{\rm Lie}.
\end{equation*}
By Lemma \ref{RemarkS6}, we have
\begin{align}
\phi_{\g}\big[e^{tD\phi_{\g}^{-1}}\phi_{\g}(x), e^{tD\phi_{\g}^{-1}}\phi_{\g}(y)\big]_{\rm Lie}&= \phi_{\g}\big[\phi_{\g}^{-1}e^{tD\phi_{\g}^{-1}}\phi_{\g}(x),\phi_{\g}^{-1}e^{tD\phi_{\g}^{-1}}\phi_{\g}(y)\big]_{\g}\nonumber\\
&= \big[e^{tD\phi_{\g}^{-1}}\phi_{\g}(x), e^{tD\phi_{\g}^{-1}}\phi_{\g}(y)\big]_{\g}.\label{eq:cond}
\end{align}
i.e.,
\begin{equation*}
\big[e^{tD\phi_{\g}^{-1}}\phi_{\g}(x), e^{tD\phi_{\g}^{-1}}\phi_{\g}(y)\big]_{\g}=\big[\Ad_{\phi_{\g}^{-1}}\Hexp(tD)(x), \Ad_{\phi_{\g}^{-1}}\Hexp(tD)(y)\big]_{\g}.
\end{equation*}
Therefore,
\begin{equation}
\label{eq:45}\Hexp(tD)[x, y]_{\g}=\big[\Ad_{\phi_{\g}^{-1}}\Hexp(tD)(x), \Ad_{\phi_{\g}^{-1}}\Hexp(tD)(y)\big]_{\g}.
\end{equation}

Conversely, we show that if $D\colon \g\rightarrow \g$ is a linear map and it satisfies~\eqref{eq:45}, then $D$ is a~derivation of the Hom-Lie algebra $(\g, [\cdot, \cdot]_{\g}, \phi_{\g})$. Note that~\eqref{eq:43} and~\eqref{eq:cond} implies that
\begin{gather*}
\phi_{\g} e^{tD\phi_{\g}^{-1}}[\phi_{\g}(x), \phi_{\g}(y)]_{\rm Lie}=\big[e^{tD\phi_{\g}^{-1}}\phi_{\g}(x), e^{tD\phi_{\g}^{-1}}\phi_{\g}(y)\big]_{\g}=\phi_{\g}\big[e^{tD\phi_{\g}^{-1}}\phi_{\g}(x), e^{tD\phi_{\g}^{-1}}\phi_{\g}(y)\big]_{\rm Lie}
\end{gather*}
for all $x,y\in \g$. Therefore, $D\phi_{\g}^{-1}$ is a derivation of the Lie algebra $(\g, [\cdot, \cdot]_{\rm Lie})$. Subsequently, from Proposition~\ref{eqv3}, it follows that $D$ is a derivation of the Hom-Lie algebra $(\g, [\cdot, \cdot]_{\g}, \phi_{\g})$.
\end{proof}

\begin{rmk}
 Consider the Hom-Lie algebra $(_q\sln_2,[\cdot,\cdot],\alpha)$ given in \cite[Example 1]{LS}, where
\begin{equation*}
[ e,f]=\frac{1+q}{2}h,\qquad [ h,e]=2e,\qquad [h,f]=-2qf
\end{equation*}
and
\begin{equation*}
\alpha(e)=\frac{q^{-1}+1}{2}e,\qquad \alpha(h)=h, \qquad \alpha(f)=\frac{q+1}{2}f.
\end{equation*}
Here $\{e, f, h\}$ is the basis of $\sln_2$. It is straightforward to deduce that
\begin{equation*}
\alpha([ e, f])=\frac{q+1}{2}h, \qquad[\alpha(e),\alpha(f)]=\frac{q^{-1}+1}{2}\frac{q+1}{2}\frac{q+1}{2}h,
\end{equation*}
which implies that $\alpha$ does not preserve $[\cdot,\cdot]$, i.e., $\alpha$ is not an algebraic morphism. So the Hom-Lie algebra $_q\sln_2$ is not a multiplicative Hom-Lie algebra. Thus the integration-differentiation approach developed in this paper does not apply to this concrete example. We will study the integration of this more general case in the future.
\end{rmk}

\subsection*{Acknowledgements}

We give our warmest thanks to the referees for very helpful suggestions that improve the paper. Research supported by NSFC (11922110)

\pdfbookmark[1]{References}{ref}
\LastPageEnding

\end{document}